\documentclass[11pt]{amsart}%
\usepackage{amssymb,amsmath,amsfonts,amsthm,array,bm,color}%
\usepackage{amsmath}%
\usepackage{amsfonts}%
\usepackage{amssymb}%
\usepackage{graphicx}
\usepackage{textgreek}
\usepackage[a4paper, left=3.5cm, right=3.5cm, top=4cm, bottom=4cm]{geometry}
\usepackage{thmtools} 
\usepackage[pdfdisplaydoctitle,colorlinks,breaklinks,urlcolor=blue,linkcolor=blue,citecolor=blue]{hyperref} 

\numberwithin{equation}{section}

\newtheorem{theorem}{Theorem}[section]

\newtheorem{definition}[theorem]{Definition}

\newtheorem{corollary}[theorem]{Corollary}

\newtheorem{lemma}[theorem]{Lemma}

\newtheorem{proposition}[theorem]{Proposition}

\theoremstyle{remark}
\newtheorem{remark}[theorem]{Remark}

\providecommand{\U}[1]{\protect\rule{.1in}{.1in}}
\def\d{{\rm d}}
\def\L{\mathcal{L}}
\def\div{{\rm div}}
\def\E{\mathbb{E}}
\def\N{\mathbb{N}}
\def\P{\mathbb{P}}
\def\R{\mathbb{R}}
\def\T{\mathbb{T}}
\def\Z{\mathbb{Z}}

\def\FC{\mathcal{FC}}
\def\eps{\varepsilon}

\newcommand{\A}{\mathcal{A}}
\newcommand{\B}{\mathcal{B}}
\newcommand{\C}{\mathbb{C}}

\DeclareMathOperator{\imm}{i}

\def\<{\langle}
\def\>{\rangle}
\newcommand{\brak}[1]{\left\langle#1\right\rangle}
\newcommand{\wick}[1]{:\mathrel{#1}:}
\newcommand{\bra}[1]{\left[#1\right]}
\newcommand{\set}[1]{\left\{#1\right\}}
\newcommand{\expt}[1]{\mathbb{E}\left[#1\right]}
\newcommand{\abs}[1]{\left|#1\right|}
\newcommand{\pa}[1]{\left(#1\right)}
\newcommand{\norm}[1]{\left\|#1\right\|}

\title[Fokker-Planck equation for dissipative stochastic 2D Euler]{Fokker-Planck equation for dissipative 2D Euler equations with cylindrical noise}

\author[F. Flandoli]{Franco Flandoli}
\address{Scuola Normale Superiore, Piazza dei Cavalieri, 7, 56126 Pisa, Italia}
\email{\href{mailto:franco.flandoli@sns.it}{franco.flandoli@sns.it}}
\urladdr{\url{http://users.dma.unipi.it/flandoli/}}

\author[F. Grotto]{Francesco Grotto}
\address{Scuola Normale Superiore, Piazza dei Cavalieri, 7, 56126 Pisa, Italia}
\email{\href{mailto:francesco.grotto@sns.it}{francesco.grotto@sns.it}}

\author[D. Luo]{Dejun Luo}
\address{RCSDS, Academy of Mathematics and Systems Science, Chinese Academy of Sciences, Beijing 100190, China, and School of Mathematical Sciences, University of the Chinese Academy of Sciences, Beijing 100049, China}
\email{\href{mailto:luodj@amss.ac.cn}{luodj@amss.ac.cn}}
\date\today

\begin{document}

\begin{abstract}
	After a short review of recent progresses in 2D Euler equations with random
	initial conditions and noise, some of the recent results are improved by exploiting
	a priori estimates on the associated infinite dimensional Fokker-Planck equation.
	The regularity class of solutions investigated here
	does not allow energy- or enstrophy-type estimates, but only bounds in
	probability with respect to suitable distributions of the initial conditions.
	This is a remarkable application of Fokker-Planck equations in infinite
	dimensions. Among the example of random initial conditions we consider Gibbsian
	measures based on renormalized kinetic energy.
\end{abstract}

\maketitle

%
\section{Introduction}

This paper is devoted to existence of solutions to the 2-dimensional stochastic Euler equation
\begin{equation}\label{stocheuler}
\d\omega +u\cdot\nabla\omega \,\d t= -\alpha \omega \,\d t+ \sqrt{2\alpha}\,\d W,
\end{equation}
where $u$ is the divergence-less velocity field and $\omega=\nabla^\perp\cdot u$ the scalar vorticity field,
$\nabla^\perp=(\partial_2,-\partial_1)$.
The equation includes a friction term and space-time additive white noise forcing.
As a preliminary, we prove an existence result for the associated Fokker-Planck equation:
this becomes necessary since solutions to \eqref{stocheuler} under cylindrical
white noise forcing exist only in certain distributional spaces where
classical energy or enstrophy estimates are not available.
Such estimates are thus replaced by probabilistic estimates, taking averages with
respect to the solution of the Fokker-Planck equation. We believe this to be a
remarkable application of recent techniques developed for Fokker-Planck
equations in infinite-dimensional spaces, a topic that has received considerable
attention in recent years.

Before we go into technical details it may be convenient to recall the
present state of the art on some classes of stochastic partial differential
equations (SPDEs) in fluid mechanics and some of the main open questions.
The literature on the topic is enormous and we
shall discuss only a very small portion of it, neglecting for instance the
recent important contributions to compressible models, see for instance
\cite{BrFeHo18}. In the more classical incompressible case there are various
reviews, like \cite{Ku06,AlFlSi08,Fl15,KuSh12}. The three chief open research directions in deterministic
incompressible fluid mechanics deal with:
\begin{itemize}
	\item[i)] well-posedness;
	\item[ii)] inviscid limits;
	\item[iii)] turbulence.
\end{itemize}
\noindent Probability has obvious relations with turbulence, whereas it is less clear
how much it can be related to the former two. Problem (ii) has been essentially left
untouched by stochastic methods, despite it being quite promising. As for problem (i),
a huge effort has been devoted to the attempt at extending and improving the
deterministic theory by means of probability and stochastic models. The
present work stems from the study of models proposed to describe features of turbulence such as the
inverse energy cascade in dimension 2, see \cite{BoEc12,Gr19},
and it may fit into the framework of question (i), since it extends the class of function spaces
in which Euler equation is solved.

The most important open problems in class (i) consist in well-posedness of
basic deterministic equations. Obviously, the outstanding one is well-posedness of 3-dimensional Navier-Stokes equations;
we refer to \cite{Fe06} for the statement of the problem in the occasion of the millennium prize declaration.
Since all equations discussed in what follows are inviscid, we do not discuss
this fundamental problem much further. However, we mention that in sight
of striking well-posedness results for stochastic ordinary differential equations
(SDEs) with very irregular drift and additive noise such as \cite{Ve80,KrRo05}, there exist the general belief that suitably non-degenerate
additive noise may regularize several classes of differential equations,
providing for instance uniqueness results in cases where the deterministic equation
may not have a unique solution.
In infinite dimensions there are important examples of such results, see for instance \cite{DPFl10, DPFlPrRo13,DPFlRoVe16}.
However, the drift terms in those works is still far from
the irregularity and unboundedness of the inertial term of 3D Navier-Stokes
equations, and requirements on the noise restrict applications to parabolic equations, involving the Laplacian operator, in dimension $d=1$.
The strategy of those papers consists in solving directly the infinite dimensional
Kolmogorov equation associated to the SPDE. In the case of 3D Navier-Stokes
equations the corresponding Kolmogorov equation has been solved in
\cite{DPDe03} but the regularity of solutions is not sufficient to deduce
uniqueness results of weak solutions to the stochastic 3D Navier-Stokes
equations. That research however was not without interesting consequences;
among others, the existence of global in time Markov selections with the Strong-Feller
property --- a striking continuous dependence on initial conditions --- which has no deterministic counterpart in the theory of 3D
Navier-Stokes equations, see \cite{FlRo08}.

In the inviscid, incompressible class, the main open problems concern the 3-dimensional
Euler equations: only local results are known, except for special
notions of solutions, see \cite{Li98,MaBe02,DLSz09}.
Such equations represent a too difficult task for a first stage understanding
of regularization by noise. Let us thus discuss the simpler case of the 2-dimensional
Euler equations on the torus $\mathbb{T}^{2}=\mathbb{R}^{2}/(2\pi \mathbb{Z})^{2}$ in vorticity form:
\begin{equation}\label{2deuler}
\begin{cases}
\partial_{t}\omega+u\cdot\nabla\omega=0,\\
\div\, u=0,\\
\omega=\nabla^\perp u.
\end{cases}
\end{equation}
When the initial condition $\omega|_{t=0}$ is bounded measurable,
a celebrated result of Judovi\v{c}, see \cite{Ju63}, establishes the existence
of a unique solution. The result has been extended to additive noise
(regular in space) in \cite{BeFl99} and to multiplicative transport type
stochastic perturbations in \cite{BrFlMa16}.

When the regularity of the initial condition $\omega|_{t=0}$ is decreased,
say to $L^{p}(\mathbb{T}^{2})$, $p\in(1,\infty)$, global existence can still be proved with arguments based on the
formal conservation of the $L^{p}$-norm of $\omega$; however, uniqueness is open,
see \cite{Li98} for a discussion.
It is therefore natural to try stochastic approaches to restore uniqueness below the class $L^{\infty}$:
unfortunately we still do not know whether there exists a noise, either additive or multiplicative,
that might do so.
This and other closely related open problems originated a considerable amount of research:
several attempts have been made to prove that suitable multipicative
transport type noises --- a natural choice in inviscid problems due to its conservation
properties --- regularize first order, transport type PDEs. The case of linear
transport equations has been understood quite well, see for instance
\cite{FlGuPr10, Ma11, Fl15, FlMaNe14, BeFlGuMa14}.
The nonlinear case is much more difficult, and only fragmentary results are available:
point vortex solutions to the 2D Euler equations are regularized \cite{FlGuPr11}; for
dyadic models and their generalizations on trees \cite{BaBiFlMo13}, uniqueness
holds thanks to multiplicative noise \cite{BaFlMo10, Bi13} and a
variant of the same technique applied also to a 3D Leray $\alpha$-model
\cite{BaBeFe14}; Hamilton-Jacobi equations \cite{GaGe19} and
scalar conservation laws \cite{GeMa18} are also regularized by suitable
multiplicative noise, although not of transport type.

The well-posedness problem (i) has another aspect which received attention, both in the
seventies and again more recently due to progresses made for nonlinear
dispersive equations: extending the existence theory to distributional
classes of vorticity fields $\omega$, for the 2D Euler equations
\eqref{2deuler}. The motivation is twofold: to understand the limits of
PDE theory in terms of roughness of solutions, and to establish a
rigorous set-up for investigation of certain explicit invariant measures of
Gibbs type (often just Gaussian), which are supported only on such
distributional spaces. We recall that invariant measures are of potential
interest for turbulence theory, see for instance the classical work \cite{Kr75},
thus all explicit examples deserve attention.

The first works in this direction are reviewed in \cite{AlFlSi08}. A
basic result is \cite{AlCr90}, proving existence of stationary solutions of
equations \eqref{2deuler} in the negative order Sobolev space $H^{-1-\delta
}(\T^2)$, $\delta>0$; the time marginal is the 2-dimensional space
white noise, also known as Enstrophy measure in this context, or the Energy-Enstrophy Gibbs measure
(which we review in \autoref{sec:energyens}).
This theory was recently revised by means
of an alternative approach based on point vortex approximation \cite{Fl18}.
These works, devoted to the deterministic equation \eqref{2deuler} with
random initial conditions, have been generalized to stochastic cases, on one
hand to the case of multipicative transport noise, see in particular
\cite{FlLu18a,FlLu17,FlLu19a}; on the other hand to the case
of additive space-time white noise and friction \cite{Gr19} (multiplicative
noise is formally conservative, while in the case of additive noise a friction
is needed to allow stationary solutions).
The 2D Euler equations with additive noise, possibly including friction, their corresponding
stationary solutions and invariant measures had already been considered
before. However, the space regularity of noise is such that solutions are
function-valued, not distributions and invariant measures are supported on spaces of functions: we refer for instance to \cite{BeFl99, BrPe01, BeFe12, BeMi12, CoGHVi14,GHVi14,BeFe14,GHSvVi15, BeFe16},
and also to other related results in \cite{Ku04,Ku06,Fl15}. Many of those models and results are inspired by the open problem of turbulence (iii);
in connection with this question and the previous references we also mention \cite{BiMo17,GrRo19,FlLu19b}.

In the stream of aforementioned works on stochastic 2D Euler equations with damping,
but in the specific regime of distributional solutions --- corresponding to the case of cylindrical white noise ---
investigated by \cite{Gr19}, we extend here that result from stationary to
non-stationary solutions, with random initial conditions related to a Gaussian
invariant measure. As already remarked, this extension is based on preliminary
estimates on the associated Fokker-Planck equation, where some technical
aspects are inspired by recent works on a different kind of noise, see \cite{FlLu18a}.
Using the method of Galerkin approximation, we shall prove existence of solutions $\rho_t$ to the Fokker-Planck equation with initial data $\rho_0$ which are $L\log L$-integrable with respect to the  white noise invariant measure $\mu$ of \eqref{stocheuler}.
In the case $\alpha>0$, the relative entropy of these solutions decrease exponentially fast as $t$ grows to $\infty$;
this together with an inequality of Kullback \cite{Kullback} implies the convergence to equilibrium of the solutions we constructed.
In the case $\rho_0\in L^2(\mu)$, we also have exponential convergence of $\rho_t$ in $L^2$-norm.
These results put forward a difficult question that we will not treat here,
namely the search for a notion of uniqueness and ergodicity of the invariant measure $\mu$,
and convergence to equilibrium of the non-stationary solutions.
Among the non-stationary initial conditions of special physical interest there is the Energy-Enstrophy
Gibbs measure associated to the renormalized energy investigated by
\cite{GrRo19} and previous works, see \cite{AlFlSi08}.

This paper is organized as follows. In Section \ref{sec:mainresult} we first recall the definition of the nonlinear term in the weak formulation of the Euler equation and the Fokker-Planck equation, and then state our main results: Theorems \ref{thm:llogl}--\ref{thm:flow}. The proofs of these results are given in Sections 3--5. In the last section, we discuss the example which takes the Energy-Enstrophy
Gibbs measure as the initial condition.

%
\section{Notation and Main Results}\label{sec:mainresult}

Consider the stochastic dissipative Euler equation in vorticity form on the 2-dimensional torus $\T^2=\R^2/(2\pi\Z)^2$:
\begin{equation}\label{ddeuler}
\begin{cases}
\d\omega+u\cdot\nabla\omega\, \d t= -\alpha \omega \,\d t+ \sqrt{2\alpha} \,\d W,\\
\nabla^\perp u=\omega,
\end{cases}
\end{equation}
where $\omega$ has zero space average on $\T^2$.
The latter gauge choice will be assumed throughout this paper, thus all function spaces on $\T^2$ are tacitly assumed
to have zero averaged elements.
We shall henceforth relate the velocity field $u$ to vorticity $\omega$ by the \emph{Biot-Savart law}
inverting the second equation in (\ref{ddeuler}):
\begin{equation}\label{biotsavart}
u= u(\omega)=K\ast\omega,\quad K=-\nabla^\perp G,
\end{equation}
with $G$ the (zero averaged) Green function of the Laplacian operator. We recall that $G(x,y)=G(x-y)$ is a translation invariant
function, smooth in all points except the origin $0$, where it has a logarithmic singularity. As a consequence, $K(x,y)=K(x-y)$ is a translation invariant vector field, exploding at $0$ with the speed $|x|^{-1}$.

The forcing term $\d W$ is the space-time white noise on $\T^2$, so that our model coincides with the one studied in \cite{Gr19},
where existence of weak (both in probabilistic and analytic sense) stationary solutions were proved by approximation with a system of Euler point vortices with creation and quenching.

The space-time white noise, a centred delta-correlated Gaussian field, is equivalently understood as the cylindrical Wiener process on $L^2(\T^2)$,
see \cite{DPZa14}. The stationary fixed time marginal considered in \cite{Gr19} is the unique invariant measure of the linear part of the equation
\begin{equation}\label{ouequation}
\d Z= -\alpha Z \,\d t+ \sqrt{2\alpha} \,\d W,
\end{equation}
that is, the space white noise measure on $\T^2$, often referred to in this context as the \emph{enstrophy measure}. This is due to the fact that it can be realized as the Gaussian measure $\mu$ on the abstract Wiener space $(H^{-1-\delta}(\T^2),L^2(\T^2))$ (any $\delta>0$), where the inner product of the Cameron-Martin space $L^2(\T^2)$ coincides with the quadratic form associated to \emph{enstrophy}, $S(\omega)=\frac12 \int_{\T^2}\omega^2 \,\d x$, which is (formally) a first integral of the 2-dimensional Euler equation.

In what follows, brackets $\brak{\cdot,\cdot}$ will denote the $L^2(\T^2)$ inner product or more generally the $L^2(\T^2)$-based
duality coupling between distributions and functions on $\T^2$.
In order to lighten notation, let us fix $\delta>0$ and denote $E=H^{-1-\delta}(\T^2)$. Moreover, we denote $\eta$
the isonormal Gaussian process on $L^2(\T^2)$, which is the space white noise on $\T^2$, $\mu$ is its law.

\subsection{The Euler nonlinear term in the Gaussian setting}

The definition of the nonlinear term in \eqref{ddeuler} when the law of $\omega_t$ is $\mu$,
or an absolutely continuous measure with respect to
$\mu$, is not immediate, and it has been thoroughly discussed in \cite{Fl18} and related works, \cite{FlLu18b,DPFlRo17, Gr19}.
We will rely upon the arguments of Subsection 2.5 of \cite{Fl18}, which we now review.

Were the stochastic processes $\omega_t$ and $W_t$ smooth (both in time and space), the differential formulation of \eqref{ddeuler}
would be equivalent to the (analytically) weak, integral formulation: for test functions $\phi\in C^\infty(\T^2)$,
\begin{equation*}
	\brak{\omega_t-\omega_0,\phi}-\int_0^t \brak{(K\ast\omega_s)\omega_s,\nabla\phi} \d s
	=-\alpha\int_0^t\brak{\omega_s,\phi} \d s+\sqrt{2\alpha}\brak{W_t,\phi}.
\end{equation*}
Since the kernel $K$ is skew-symmetric, we have
\begin{equation}\label{symmetrization}
	\brak{(K\ast\omega_s)\omega_s,\nabla\phi}= \brak{\omega_s\otimes \omega_s, H_\phi},
	\quad H_\phi(x,y):=\frac{\nabla \phi(x)-\nabla \phi(y)}{2}\cdot K(x-y).
\end{equation}
The key fact here is that $H_\phi$ is a bounded symmetric function, smooth outside the diagonal set $\set{(x,x)\in \T^2\times \T^2}$,
where it has a jump discontinuity: this is easily seen by means of Taylor expansion.
A lengthy but elementary computation in Fourier series reveals that the Sobolev regularity of $H_\phi$ is at best $H^{2-}(\T^2 \times \T^2)$, thus the above symmetrized formulation, dating back to the works of Delort and Schochet, see \cite{Sc95},
allows us to give a proper meaning to \eqref{ddeuler} in the case when $\omega_t\in H^{-1+}(\T^2)$.
Here comes into play the essential contribution of Gaussian distributions.

\begin{proposition}\label{prop:nonlinearity}
	Let $\phi\in C^\infty(\T^2)$ and $\omega$ be a random distribution on $\T^2$ with law $\rho\, \d \mu$, $\rho\in L^p(E,\mu)$
	for some $p>1$. For any sequence $(H^n_\phi)_{n\in\N}\subset C^\infty(\T^2 \times \T^2)$ of symmetric functions such that
	\begin{align}
		L^2(\T^2 \times \T^2)-\lim_{n\to \infty} H^n_\phi&=H_\phi, \label{eq:approxH}\\
		\lim_{n\to \infty} \int_{\T^2} H^n_\phi(x,x) \d x&=0, \label{eq:approxH-1}
	\end{align}
	the limit
	\begin{equation}\label{doubleintegral}
		\brak{\omega\otimes \omega, H_\phi}:= \lim_{n\to \infty} \brak{\omega\otimes\omega, H^n_\phi}
	\end{equation}
	exists in $L^1(\mu)$ and it does not depend on the approximating sequence $H^n_\phi$ among the ones satisfying the above properties.
	Moreover,
	\begin{equation}\label{goodbound.0}
		\E\Big[\big| \brak{\omega\otimes \omega, H^n_\phi- H_\phi} \big|\Big] \leq C_p \norm{H^n_\phi - H_\phi}_{L^2(\T^2\times \T^2)}^{1/p'} + \bigg| \int_{\T^2} H^n_\phi(x,x) \d x \bigg| ,
	\end{equation}
	with $\frac1p+\frac1{p'}=1$, and for any $q\in [1,\infty)$ it holds
	\begin{equation}\label{goodbound}
		\expt{|\brak{\omega\otimes \omega, H_\phi}|^q}  \leq C_q \norm{\rho}_{L^p(E,\mu)} \norm{\phi}^q_{C^2(\T^2)},
	\end{equation}
	with $C_q$ a constant depending only on $q$.
	
	If $\rho_t \in L^\infty\pa{[0,T], L^p(E,\mu)}$ and $\omega_t$ is a process with trajectories in $C\pa{[0,T],E}$
	and marginals $\omega_t\sim \rho_t\, \d\mu$ (in particular we are assuming $\int_{E}\rho_t \, \d\mu=1$ for all $t$),
	the sequence of real processes $\brak{\omega_t\otimes\omega_t, H^n_\phi}$ converges in $L^1\pa{E,\mu; L^1([0,T])}$
	to a process $\brak{\omega_t\otimes\omega_t, H_\phi}$ which does not depend on the approximations $H^n_\phi$ as above.	
\end{proposition}

It is worth noticing that the approximations $H^n_\phi$ as in \eqref{eq:approxH} can always be obtained by regularizing the kernel $K$
in the definition of $H_\phi$. The proof of the above Proposition is detailed in \cite[Section 2.5]{Fl18}; \eqref{goodbound.0} follows easily from the proof of \cite[Theorem 14]{Fl18}.
We also remark that if $\omega\sim\mu$, the limit \eqref{doubleintegral} coincides with the  double Wiener-It\^o integral
of the kernel $H_\phi$ on the Gaussian Hilbert space $(E,\mu)$ (see \cite{Gr19} for a discussion).

With \autoref{prop:nonlinearity} at hand, we are able to give meaning to \eqref{ddeuler} and the associated Fokker-Planck equation
when the law of fixed time marginals is (absolutely continous with respect to) the space white noise measure $\mu$.

\subsection{Weak solutions to Euler and Fokker-Planck Equations}

On the torus $\T^2=\R^2/(2\pi\Z)^2$ we consider the normalized Haar measure $\d x$ such that $\int_{\T^2} \d x=1$,
and the orthonormal Fourier basis $e_k(x)=e^{\imm k\cdot x}$ of $L^2(\T^2, \d x)$.
In fact, we will only deal with real-valued objects:
Fourier coefficients of opposite modes will henceforth be complex conjugated.
Moreover, we will tacitly assume the zero average setting, that is, $0$-th Fourier modes of functions and distributions
are always null.

Let $\FC_b$ be the linear space of \emph{cylinder functions} of the form
\begin{equation*}
\varphi(\omega)=f(\hat\omega_{k_1},\dots, \hat\omega_{k_n}), \quad k_1,\dots, k_n\in \Z^2_0,
\end{equation*}
with $n\geq 1$ and $f\in C^\infty_b(\R^n)$, the space of bounded functions with bounded derivatives of all orders.
The infinitesimal generator associated to \eqref{ouequation} is $\alpha\L$, with $\L$
the generator of the \emph{Ornstein-Uhlenbeck semigroup} acting on cylinder functions as
\begin{equation*}
\L \varphi(\omega)=\sum_{i=1}^n \partial_{ii}f(\hat\omega_{k_1},\dots, \hat\omega_{k_n})
-\sum_{i=1}^{n}\partial_i f(\hat\omega_{k_1},\dots, \hat\omega_{k_n})\hat\omega_{k_i}.
\end{equation*}

The infinitesimal generator associated to \eqref{ddeuler} can be written formally as
\begin{equation}\label{generator}
\A \varphi(\omega)=\B\varphi(\omega) + \alpha \L \varphi(\omega),\quad \B\varphi(\omega)=-\brak{(K\ast \omega)\cdot\nabla\omega, D\varphi(\omega)},
\end{equation}
whose action on cylinder functions $\varphi\in\FC_b$ is given in terms of $\L$ defined above and
\begin{equation*}
D\varphi(\omega)=\sum_{i=1}^{n}\partial_i f(\hat\omega_{k_1},\dots, \hat\omega_{k_n})e_{k_i}.
\end{equation*}
To give a rigorous definition of the \emph{Liouville operator $\B$ of Euler equation \eqref{2deuler}}, we make use of \autoref{prop:nonlinearity}
(see also the discussion in \cite{DPFlRo17}). First, we combine the latter two expressions with \eqref{symmetrization} to obtain
\begin{align*}
	\B\varphi(\omega)&=-\sum_{i=1}^{n}\partial_i f(\hat\omega_{k_1},\dots, \hat\omega_{k_n})\brak{(K\ast \omega)\cdot\nabla\omega, e_{k_i}},\\
          &=\sum_{i=1}^{n}\partial_i f(\hat\omega_{k_1},\dots, \hat\omega_{k_n}) \brak{\omega\otimes\omega, H_{e_{k_i}} };
\end{align*}
hence, by \autoref{prop:nonlinearity}, we can define the real random variable $\B\varphi(\eta)\in L^1(\mu)$
for all cylinder functions $\varphi\in\FC_b$. As already observed above, $\brak{\eta\otimes \eta, H_\phi}$
is in fact an element of the second Wiener chaos of the Gaussian process $\eta$, since it coincides with the double It\^o-Wiener integral.
As a consequence, $\B$ is exponentially integrable when acting on cylinder functions:
\begin{equation}\label{exponentialint}
	\expt{\exp\pa{\varepsilon |\B\varphi(\eta)|}}<\infty \quad \mbox{for all small } \eps >0
\end{equation}
(see \cite[Theorem 8]{DPFlRo17} for an explicit computation).

The singularity of the nonlinear term is such that the operator $\B$, regarded as a vector field acting
as a derivation on the whole $E$, does not take values in the Cameron-Martin space $L^2(\T^2)$, or even in $E$, see \cite{AlCr90}.
Nonetheless, it formally holds $\div_\mu \B=0$, in agreement with the fact that there exist stationary solutions of Euler equation, and thus
of (\ref{ddeuler}), with invariant measure $\mu$ (see \autoref{sec:energyens} below).
We also notice that $\B$ is skew-symmetric on $\FC_b$, as revealed by direct computation.

Let us consider the Fokker-Planck equation associated to \eqref{ddeuler}:
\begin{equation}\label{fpe}
\begin{cases}
\partial_t \rho = \A^\ast \rho =-\B\rho+ \alpha \L \rho,\\
\rho|_{t=0} =\rho_0.
\end{cases}
\end{equation}

\begin{definition}\label{def:weaksolfp}
	Given $\rho_0\in L^1(\mu)$, for any $\alpha\geq 0$, we say that $\rho\in L^1_{loc} \big(\R_+, L^1( E,\mu) \big)$ is a weak solution of the Fokker-Planck equation \eqref{fpe} if
	\begin{itemize}
		\item[\rm (a)] for any $\varphi\in \FC_b$ and $T>0$,
		$$\int_0^T  \int_E |\rho_t \A\varphi|\, \d\mu \d t <\infty; $$
		\item[\rm (b)]  for any $f\in C_c^1(\R_+ )$ and $\varphi\in \FC_b$ it holds
		\begin{equation}\label{fpe-weak}
		f(0) \int_E \rho_0 \varphi\, \d\mu + \int_0^\infty \! \int_E f'(t) \rho_t \varphi\, \d\mu \d t
		+ \int_0^\infty \! \int_E f(t) \rho_t \A\varphi \, \d\mu \d t=0.	
		\end{equation}
	\end{itemize}
\end{definition}

\begin{remark}
	Identity \eqref{fpe-weak} implies that, in the distributional sense,
	\begin{equation*}
	\frac{\d}{\d t} \int_E \rho_t \varphi\, \d\mu = \int_E \rho_t \A\varphi\, \d\mu\quad \mbox{for a.e. } t\in (0,\infty).
	\end{equation*}
	Since the right-hand side is locally integrable in $t\in (0,\infty)$, the map $[0,\infty) \ni t\mapsto \int_E \rho_t \varphi\,
    \d\mu$ is absolutely continuous, thus $\rho_t$ is weakly continuous in time. This also gives meaning to the initial condition specification $\rho|_{t=0} =\rho_0$.
	Moreover, taking $\varphi \equiv 1$ yields $\int_E \rho_t\, \d\mu = \int_E \rho_0\, \d\mu$ for all $t>0$.
\end{remark}

Our first aim is to prove existence results of the Fokker-Planck
equation with general initial conditions, using a Galerkin type scheme.

\begin{theorem}\label{thm:llogl}
	Let $\rho_{0}\in L\log L(E,\mu; \R_+)$ and $\alpha\geq 0$. Then,
	\begin{enumerate}
		\item[\rm(i)] there exists a weak solution $(\rho_{t})_{t\in \R_+ }$ of the
		Fokker-Planck equation \eqref{fpe} in the sense of \autoref{def:weaksolfp};
		\item[\rm(ii)] for almost every $t>0$ it holds
		\begin{equation*}
		\int_E \rho_t \log \rho_t\, \d\mu \leq e^{-2\alpha t} \int_E \rho_0 \log \rho_0\, \d\mu + \big(1- e^{-2\alpha t}\big) \|\rho_0\|_{L^1} \log \|\rho_0\|_{L^1}.
		\end{equation*}
	\end{enumerate}
	In particular, if $\rho_0$ is a probability density and $\alpha>0$, then the relative entropy of the weak solution $\rho_t$ decreases exponentially fast, which in turn implies the convergence to equilibrium of $\rho_t$: for almost every $t>0$ it holds
	\begin{equation*}
		\norm{\rho_t-1}_{L^1}\leq e^{-\alpha t}\sqrt{2 \int_E \rho_0 \log \rho_0\d\mu}.
	\end{equation*}
\end{theorem}

The last assertion is an immediate consequence of the exponential decay of entropy and Kullback's inequality, see \cite[(11)]{Kullback}.  Next, we deduce also an existence result for $L^p\, (p>1)$ initial densities. We state it explicitly since it will play an important role in building solutions to the stochastic equation \eqref{ddeuler}.

\begin{theorem}\label{thm:lp}
	Let $\rho_{0}\in L^p(E,\mu)$ with $p>1$ and $\alpha\geq 0$. Then,
	\begin{enumerate}
		\item[\rm(i)] there exists a weak solution $\rho \in L^\infty\big( \R_+, L^p(E,\mu)\big)$ to Fokker-Planck equation \eqref{fpe}
		in the sense of \autoref{def:weaksolfp};
		\item[\rm(ii)] if $p=2$, then, denoting by $\bar \rho_0 = \int_E \rho_0\, \d\mu$, we have, for a.e. $t>0$,
		\begin{equation*}
		\norm{\rho_t - \bar \rho_0}_{L^2} \leq e^{-\alpha t} \norm{\rho_0 - \bar \rho_0}_{L^2}.
		\end{equation*}
	\end{enumerate}
\end{theorem}

Our second result is the existence of weak (both in probabilistic and analytical sense)
solutions to the Euler equation \eqref{ddeuler} in the setting of \autoref{thm:lp}.

\begin{theorem}\label{thm:flow}
	Let $p>1$, $\alpha\geq 0$, $T>0$.
	Assume that $\rho_{0}\in L^p(E,\mu; \R_+)$ is a probability density,
	and let $\rho \in L^\infty\big( 0,T; L^p(E,\mu)\big)$ be a weak solution
	obtained in \autoref{thm:lp} to Fokker-Planck equation \eqref{fpe} with initial datum $\rho_{0}$.
	There exist a filtered probability space on which a cylindrical Wiener process $W$ on $L^2(\T^2)$
	and an adapted process $\omega_t$ are defined such that
	\begin{enumerate}
		\item[\rm(i)] $\omega\in C([0,T],E)$ with probability one;
		\item[\rm(ii)] for almost every $t\in [0,T]$, $\omega_{t}$ has law $\rho_{t}\, \d\mu$;
		\item[\rm(iii)] for any $\phi\in C^\infty(\T^2)$ and $t\in [0,T]$,
		\begin{equation*}
		\brak{\omega_t,\phi}=\brak{\omega_0,\phi}+ \int_0^t \brak{\omega_s \otimes \omega_s, H_\phi} \d s -\alpha \int_0^t \brak{\omega_s,\phi} \d s +\sqrt{2\alpha} \brak{W_t,\phi},
		\end{equation*}
		where the nonlinear term is defined as in \autoref{prop:nonlinearity}.
	\end{enumerate}
\end{theorem}

%
\section{The Galerkin approximation and \texorpdfstring{$L\log L$}{LlogL} initial data}

Let us define the finite-dimensional projection of $H=L^2(\T^2, \d x)$ onto the finite set of modes $\Lambda_N=\set{k\in\Z^2_0:|k|_\infty\leq N}$,
\begin{equation}
\Pi_N:H\ni f\mapsto \Pi_N f=\sum_{k\in\Lambda_N} \brak{f, e_k}_H e_k\in H_N,
\end{equation}
where we can identify the finite dimensional codomain with
\begin{equation}
H_N=\set{\xi\in \C^{\Lambda_N}:\bar \xi_k=\xi_{-k}}
\end{equation}
(whose dimension is $|\Lambda_N|$).
On $H_N$ we consider the Euclidean inner product induced by $\C^{\Lambda_N}$, and the Gaussian measure $\mu_N$ having Fourier coefficients
$\hat\mu_N(k)=\overline{\hat\mu_N(-k)}$ with the law of independent standard complex Gaussian distributions.

We consider the following Galerkin approximation of \eqref{ddeuler}:
\begin{equation}
\d\Pi_N\omega+\Pi_N((K\ast\Pi_N\omega)\cdot \nabla \Pi_N\omega) \d t=-\alpha \Pi_N\omega \, \d t +\sqrt{2\alpha} \,\d \Pi_N W.
\end{equation}
This equation is in fact an SDE in $\omega^N\in H_N$, and it can be rewritten as
\begin{equation}\label{sde}
\d\omega^N +b_N \big(\omega^N \big)\d t=-\alpha \omega^N \d t+\sqrt{2\alpha}\, \d W^N, \quad W^N=\sum_{k\in\Lambda_N} W^k e_k,
\end{equation}
where the $W^k$'s are independent standard complex Brownian motions such that $\overline{W^k}=W^{-k}$,
and the drift is given by
\begin{equation*}
b_N(\xi)=-\sum_{n\in \Lambda_N} e_n \sum_{k\in \Lambda_N} \frac{k^\perp \cdot n}{|k|^2} \xi_k\xi_{n-k}, \quad \xi\in H_N,
\end{equation*}
as one can prove by a straightforward computation in Fourier series using that $K(x)=\sum_{k\in\Z^2_0} \frac{\imm k^\perp}{|k|^2}e_k(x)$.
By means of the above expression, it is easy to check that, for all $\xi\in H_N$,
\begin{equation}\label{propb}
\brak{b_N(\xi), \xi}_{H_N}=0, \quad \div_{\mu_N}b_N(\xi)=\div\, b_N(\xi)-\brak{b_N(\xi), \xi}_{H_N}=0.
\end{equation}

The SDE \eqref{sde} has smooth coefficients, so there exists a unique strong local solution $\omega^N_t$ given an initial datum $\omega^N_0\in H_N$; the forthcoming estimate shows that it is also global in time.

\begin{lemma}
	If $\omega^N_t$ is a solution of \eqref{sde}, then, for any $t\geq 0$,
	\begin{equation*}
	\expt{\abs{\omega^N_t}_{H_N}^2} \leq \abs{\omega^N_0}_{H_N}^2 e^{-2\alpha t} + \abs{\Lambda_N} (1 -e^{-2\alpha t}).
	\end{equation*}
\end{lemma}

\begin{proof}
	By the It\^o formula and \eqref{propb}, and omitting all subscripts $H_N$,
	\begin{align*}
	\d\abs{\omega^N_t}^2 &= -2\brak{\omega^N_t, b_N\big(\omega^N_t \big)+\alpha\omega^N_t}\d t+2\sqrt{2\alpha} \brak{\omega^N_t,\d W^N_t}+2\alpha\brak{\d W^N_t, \d W^N_t}\\
	&=-2\alpha \abs{\omega^N_t}^2 \d t+2\sqrt{2\alpha}\brak{\omega^N_t,\d W^N_t}+2\alpha\abs{\Lambda_N} \d t,
	\end{align*}
	and therefore
	\begin{equation*}
	\d\pa{e^{2\alpha t}\abs{\omega^N_t}^2}=2\sqrt{2\alpha}\, e^{2\alpha t}\brak{\omega^N_t,\d W^N_t}+2\alpha e^{2\alpha t}\abs{\Lambda_N} \d t.
	\end{equation*}
	If we define, for $R>0$, the stopping time
	\begin{equation*}
	\tau_R=\inf\set{t>0: \abs{\omega^N_t} \geq R},
	\end{equation*}
then we have
	\begin{align*}
	\expt{e^{2\alpha (t\wedge \tau_R)}\abs{\omega^N_t}^2}
	&=\abs{\omega^N_0}^2+2\sqrt{2\alpha}\, \expt{\int_0^{t\wedge \tau_R}e^{2\alpha s}\brak{\omega^N_s,\d W^N_s}}\\
	&\quad +\abs{\Lambda_N} \E\big(e^{2\alpha (t\wedge \tau_R)} -1 \big)\\
	&\leq \abs{\omega^N_0}^2+\abs{\Lambda_N}(e^{2\alpha t} -1),		
	\end{align*}
	which concludes the proof if we let $R\uparrow\infty$ by Fatou's lemma.
\end{proof}

\subsection{Finite dimensional Fokker-Planck equation}

Let $\L_N$ be the Ornstein-Uhlenbeck operator on $H_N$; then $\alpha\L_N$ is the infinitesimal generator of the linear part of \eqref{sde}.
We can introduce the Galerkin approximation $\A_N$ of $\A$, acting on smooth functions $F\in C_b^2(H_N)$ as
\begin{equation}\label{approxgenerator}
\A_N F(\xi)= -\brak{b_N(\xi), \nabla F(\xi)}_{H_N} +\alpha \L_N F(\xi).
\end{equation}
We can thus write the Fokker-Planck equation corresponding to \eqref{sde}:
if the law of $\omega^N_0$ has a smooth probability density $\rho^N_0$ (with respect to $\mu_N$),
so does $\omega^N_t$ for any later time, and the density $\rho^N_t$ satisfies
\begin{equation}\label{finitefp}
\begin{cases}
\partial_t \rho_t^N=\A_N^* \rho_t^N,\\
\rho^N|_{t=0} = \rho_0^N.
\end{cases}
\end{equation}

\begin{remark}\label{densityestimate}
	Simple heuristic arguments immediately give rise to an a priori estimate on the entropy of $\rho_t^N$.
	Indeed, if $\rho^N_t$ is a smooth solution of \eqref{finitefp}, for any $t\geq 0$,
	\begin{align*}
	\partial_t \big(\rho^N_t\log \rho^N_t\big) &= \big(1 + \log \rho^N_t\big) \partial_t \rho^N_t\\
	&= \big(1 + \log \rho^N_t\big) \big\< b_N, \nabla \rho^N_t\big>_{H_N} + \alpha \big(1 + \log \rho^N_t\big) \L_N \rho^N_t.
	\end{align*}
	Integrating on $H_N$ with respect to $\mu_N$ and using \eqref{propb} we get
	\begin{equation*}
	\int_{H_N} \rho^N_t\log \rho^N_t\d\mu_N + \alpha \int_0^t \int_{H_N} \frac{\abs{\nabla \rho^N_s}^2}{\rho^N_s}\d\mu_N \d s = \int_{H_N} \rho^N_0\log \rho^N_0\d\mu_N.
	\end{equation*}
	However, the above computation is somewhat formal, since the drift $b_N$ has quadratic growth. In the following we give a more rigorous proof of the a priori estimate, and at the same time give a meaning to the equation \eqref{finitefp}.
\end{remark}

In the remainder of this subsection, we fix $N\in\N$ and assume that the initial condition of \eqref{finitefp} belongs to
\begin{equation}\label{initial-condition}
\rho_0^N\in L^\infty(H_N, \R_+).
\end{equation}
One can extend the result below to more general initial data,
but since the study of \eqref{finitefp} is only an intermediate step, we do not pursue such generality here.
Consider cut-off functions $\chi_n(\xi)= \chi(\xi/n)$,  $n\geq 1$, where $\chi\in C_c^\infty(H_N, [0,1])$ is a radial function (i.e., $\chi(\xi)= \chi(|\xi|_{H_N})$ by a slight abuse of notation) such that $\chi|_{B_N(1)} \equiv 1$ and $\chi|_{B_N(2)^c} \equiv 0$, $B_N(r)$ being the ball in $H_N$ centered at the origin with radius $r>0$. Define
$$b_N^{(n)}(\xi) = \chi_n(\xi) b_N(\xi), \quad \xi\in H_N, n\in \N;$$
then $b_N^{(n)}$ is a smooth vector field on $H_N$ with compact support for any $n\in \N$.
Notice that $b_N^n$ is still divergence-free since by \eqref{propb} and $\nabla \chi_n(\xi) = \chi'\big( \frac{|\xi|} n \big) \frac{\xi}{n |\xi|}$ one has
\begin{equation}\label{divergence}
	\div_{\mu_N} \big(b_N^{(n)}\big) = \div_{\mu_N} \big( \chi_n b_N\big) = \chi_n \div_{\mu_N} (b_N ) - \brak{b_N, \nabla \chi_n}_{H_N} =0.
\end{equation}

Now we consider the approximating operators
$$\A_N^{(n)} F(\xi)= -\brak{b_N^{(n)}(\xi), \nabla F(\xi)}_{H_N} +\alpha \L_N F(\xi)$$
and the corresponding Fokker-Planck equations
\begin{equation}\label{finitefp-1}
\begin{cases}
\partial_t \rho_t^{(n)}= \big(\A_N^{(n)} \big)^\ast \rho_t^{(n)},\\
\rho^{(n)}|_{t=0} = \rho_0^{(n)}= P^N_{1/n} \rho_0^N,
\end{cases}
\end{equation}
where the initial datum is regularized by means of the Ornstein-Uhlenbeck semigroup $P^N_t=e^{t\L_N}$ on $H_N$: for $t\geq 0$
the latter is explicitly given by
\begin{equation}\label{mehler}
P^N_t \rho_0^N(\xi) = \int_{H_N} \rho_0^N(\eta) \big[2\pi \big(1- e^{-2t} \big)\big]^{-|\Lambda_N|/2} \exp\bigg(-\frac{|\eta- e^{-t}\xi|^2}{2(1- e^{-2t}) } \bigg) \d\eta.
\end{equation}

\begin{lemma}\label{lem-regularized-fpe}
	For any $n\geq 1$, $\rho_0^{(n)} \in C_b^\infty(H_N,\R_+)$ and
	\begin{equation}\label{lem-regularized-fpe-1}
	\int_{H_N} \rho_0^{(n)} \log \rho_0^{(n)}\d\mu_N \leq \int_{H_N} \rho_0^{N} \log \rho_0^{N}\d\mu_N.
	\end{equation}
	Moreover, the solutions $\rho_t^{(n)}$ of the equations \eqref{finitefp-1} satisfy
	\begin{align}\label{lem-regularized-fpe-2}
		\sup_{t\geq 0}\norm{\rho_t^{(n)}}_\infty & \leq \norm{\rho_0^{N}}_\infty, \\
		\label{lem-regularized-fpe-3}
		\int_{H_N} \rho_t^{(n)} \log \rho_t^{(n)}\d\mu_N
		&\leq e^{-2\alpha t} \int_{H_N} \rho_0^{N} \log \rho_0^{N}\d\mu_N \\
		\nonumber
		&\quad + \big(1- e^{-2\alpha t} \big) \big\| \rho_0^N \big\|_{L^1(\mu_N) } \log \big\| \rho_0^N \big\|_{L^1(\mu_N) }
		\quad \forall\, t\geq 0.
	\end{align}
\end{lemma}

\begin{proof}
	The first assertion follows from \eqref{initial-condition} and \eqref{mehler};
	the estimate \eqref{lem-regularized-fpe-1} is a consequence of Jensen's inequality and the invariance of $\mu_N$ for the semigroup $\big(P^N_t \big)_{t\geq 0}$.
	
	Inequality \eqref{lem-regularized-fpe-2} follows from \eqref{initial-condition} and the representation
	\begin{equation*}
		\rho_t^{(n)}(\xi) = \expt{\rho_0^{(n)} \big(X^{(n)}_t \big)},
	\end{equation*}
	where $X^{(n)}_t$ is the solution to the SDE
	\begin{equation*}
		\d X^{(n)}_t= b_N^{(n)}\big(X^{(n)}_t \big)\d t - \alpha X^{(n)}_t \d t + \sqrt{2\alpha}\, \d W^N_t,\quad X^{(n)}_0= \xi.
	\end{equation*}
	
	Thanks to \eqref{divergence}, the arguments in Remark \ref{densityestimate} are now rigorous and we have
	\begin{equation}\label{lem-regularized-fpe-4}
	\frac{\d}{\d t} \int_{H_N} \rho_t^{(n)} \log \rho_t^{(n)}\d\mu_N = - \alpha \int_{H_N} \frac{\big| \nabla \rho_t^{(n)} \big|^2} {\rho_t^{(n)}} \d\mu_N.
	\end{equation}
	Recall the log-Sobolev inequality on the finite-dimensional Gaussian space $(H_N,\mu_N)$:
	\begin{equation*}
		\int_{H_N} \varphi^2 \log\frac{\varphi^2}{\|\varphi \|_{L^2(\mu_N) }^2} \d\mu_N
		\leq 2 \int_{H_N} |\nabla \varphi|^2 \d\mu_N,
		\quad \forall\, \varphi\in W^{1,2}(H_N,\mu_N).
	\end{equation*}
	Taking $\varphi = \big(\rho_t^{(n)} \big)^{1/2}$ yields
	$$\int_{H_N} \rho_t^{(n)} \log\frac{\rho_t^{(n)}}{\big\| \rho_t^{(n)} \big\|_{L^1(\mu_N) }} \d\mu_N \leq \frac12 \int_{H_N} \frac{\big| \nabla \rho_t^{(n)} \big|^2} {\rho_t^{(n)}} \d\mu_N.$$
	Combining the latter inequality with \eqref{lem-regularized-fpe-4} we obtain
	\begin{align*}
		\frac{\d}{\d t} \int_{H_N} \rho_t^{(n)} \log \rho_t^{(n)}\d\mu_N
		&\leq - 2\alpha\int_{H_N} \rho_t^{(n)} \log \rho_t^{(n)}\d\mu_N\\
		&\quad +2\alpha \big\| \rho_0^N \big\|_{L^1(\mu_N) } \log \big\| \rho_0^N \big\|_{L^1(\mu_N) },
	\end{align*}
	where we have used the fact that
	\begin{equation*}
		\big\| \rho_t^{(n)} \big\|_{L^1(\mu_N) } = \big\| \rho_0^{(n)} \big\|_{L^1(\mu_N) } = \big\| \rho_0^N \big\|_{L^1(\mu_N) }
		\quad \forall\, t>0.
	\end{equation*}
	Integrating in time, we conclude that
	\begin{align*}
		\int_{H_N} \rho_t^{(n)} \log \rho_t^{(n)}\d\mu_N
		&\leq e^{-2\alpha t} \int_{H_N} \rho_0^{(n)} \log \rho_0^{(n)}\d\mu_N\\
		&\quad + \big(1- e^{-2\alpha t} \big) \big\| \rho_0^N \big\|_{L^1(\mu_N) } \log \big\| \rho_0^N \big\|_{L^1(\mu_N) },
	\end{align*}
	which, together with \eqref{lem-regularized-fpe-1}, leads to the final result.
\end{proof}

\begin{corollary}\label{cor-1}
 Let $\rho_0^N\in L^\infty(H_N, \R_+)$. There exists a nonnegative function $\rho^N \in L^\infty\big(\R_+, L^\infty(H_N,\mu_N) \big)$
 satisfying
 \begin{align}\label{boundedness-finite-dim}
 \sup_{t\in [0,\infty)} \big\| \rho_t^N \big\|_{L^\infty(\mu_N)} &\leq \big\| \rho_0^N \big\|_{L^\infty(\mu_N)},\\
 \label{a-priori-estimate-1}
 \int_{H_N} \rho_t^N \log \rho_t^N \d\mu_N
 \leq &\ e^{-2\alpha t} \int_{H_N} \rho_0^{N} \log \rho_0^{N}\d\mu_N \\ \nonumber
 &\, + \big(1- e^{-2\alpha t} \big) \big\| \rho_0^N \big\|_{L^1(\mu_N) } \log \big\| \rho_0^N \big\|_{L^1(\mu_N) }
 \end{align}
 for almost every $t>0$; moreover, for any $f\in C_c^1(\R_+)$ and $\psi\in C_b^\infty(H_N)$,
 \begin{align}
 	\label{finitefp-2}
 	0&= f(0)\int_{H_N} \psi \rho_0^N\d\mu_N \\ \nonumber
 	&\quad +\int_0^\infty \int_{H_N}\rho_t^N \Big[ f^\prime(t) \psi + f(t) \brak{b_N, \nabla \psi}_{H_N}
 	+ \alpha f(t) \L_N\psi \Big] \d\mu_N\d t.
 \end{align}
 In particular, the above equation shows that $\rho^N$ satisfies \eqref{finitefp} in a weak sense.
\end{corollary}

\begin{remark}
	It might be possible to give a strong (i.e. pointwise) meaning to \eqref{finitefp}, but the weak sense here,
	combined with the estimate \eqref{a-priori-estimate-1}, is enough to prove existence of solutions to Fokker-Planck equations
	in the infinite dimensional case.
\end{remark}

\begin{proof}
	Thanks to \eqref{lem-regularized-fpe-2}, we can find a subsequence $\big\{ \rho^{(n_i)}\big\}_{i\in \N}$ weakly-$\ast$ converging
	in $L^\infty\big(\R_+, L^\infty(H_N,\mu_N) \big)$ to some $\rho^N$ satisfying \eqref{boundedness-finite-dim}.
	
	Now we fix any $T>0$. We know that $\rho^{(n_i)}$ also converges weakly in $L^1\big([0,T]\times H_N \big)$ to $\rho^N$.
	The sequence $\big\{ \rho^{(n_i)}\big\}_{i\in \N}$ is contained in the set
	\begin{equation*}
		\mathcal S = \bigg\{u\in L^1\big([0,T]\times H_N \big):
		u_t\geq 0,\ \int_{H_N} u_t\log u_t \, \d\mu_N \leq \Lambda(t) \mbox{ for all } t\in [0,T] \bigg\},
	\end{equation*}
	where we write $\Lambda(t)$ for the right hand side of \eqref{a-priori-estimate-1}. The convexity of the function $s\mapsto s\log s$ implies that $\mathcal S$ is a convex subset of $L^1\big([0,T]\times H_N \big)$. Since the weak closure of $\mathcal S$ coincides with the strong one, there exists a sequence of functions $u^{(n)}\in \mathcal S$ which converge strongly to $\rho^N$ in $L^1\big([0,T]\times H_N \big)$. Up to a subsequence, $u^{(n)}$ converge to $\rho^N$ almost everywhere, thus Fatou's lemma and \eqref{lem-regularized-fpe-3} implies that \eqref{a-priori-estimate-1} holds for a.e. $t\in (0,T)$. The arbitrariness of $T>0$ implies that it holds for a.e. $t\in (0,\infty)$.
	
	Finally, multiplying both sides of \eqref{finitefp-1} (with $n$ replaced by $n_i$) by $f\in C_c^1(\R_+)$ and $\psi\in C_b^\infty(H_N)$, and integrating by parts leads to
	\begin{align*}
		0=&\ f(0)\int_{H_N} \psi \rho_0^{(n_i)} \d\mu_N\\
		&\, + \int_0^\infty \int_{H_N}\rho_t^{(n_i)} \Big[ f^\prime(t) \psi + f(t) \brak{b_N^{(n_i)}, \nabla \psi}_{H_N} + \alpha f(t) \L_N\psi \Big]\d\mu_N\d t.
	\end{align*}
	Recall that $b_N^{(n)} = \chi_{n} b_N$; it is clear that $\brak{b_N^{(n)}, \nabla \psi}_{H_N}$ converges strongly to $\brak{b_N, \nabla \psi}_{H_N}$ in $L^2(\mu_N)$. By the weak-$\ast$ convergence of $\rho^{(n_i)}$, letting $i\to \infty$ yields \eqref{finitefp-2}.
\end{proof}

\subsection{Proof of \autoref{thm:llogl}}

We assume that $\rho_0\in L\log L(E,\mu; \R_+)$. Define
\begin{equation}\label{initial-data-1}
\rho^N_0 = P^N_{1/N} \expt{\rho_0\wedge N |\Pi_N}, \quad N\in \N,
\end{equation}
where $\E[\cdot |\Pi_N]$ is the conditional expectation with respect to the sub-$\sigma$-algebra generated by coordinates in $H_N$.
Note that, for any $f\in L^1(\mu)$ and all $N\geq 1$, we can regard $\expt{f| \Pi_N}$ as a function on $E$. By the invariance of $\mu_N$ for the Ornstein-Uhlenbeck semigroup $\big(P^N_t \big)_{t\geq 0}$ and Jensen's inequality,
\begin{align*}
	\int_{H_N} \rho_0^N \log \rho_0^N \d\mu_N
	&\leq \int_{H_N} \expt{\rho_0\wedge N |\Pi_N} \log \expt{\rho_0\wedge N |\Pi_N} \d\mu_N \\
	&= \int_E \expt{\rho_0\wedge N |\Pi_N} \log \expt{\rho_0\wedge N |\Pi_N} \d\mu.
\end{align*}
Using again Jensen's inequality, for all $N\in \N$,
\begin{equation}\label{entropy-init}
\int_{H_N} \rho_0^N \log \rho_0^N \d\mu_N \leq \int_E (\rho_0\wedge N) \log (\rho_0\wedge N) \d\mu \leq \int_E \rho_0 \log \rho_0\, \d\mu.
\end{equation}
Moreover, it is easy to see that
\begin{equation}\label{L1-norm}
\big\|\rho_0^N \big\|_{L^1(\mu_N)} \leq \|\rho_0\|_{L^1(\mu)}.
\end{equation}

For any $N\geq 1$, taking $\rho_0^N$ as the initial value, by the arguments in the last subsection,
we have a nonnegative solution $\rho^N$ to the finite dimensional Fokker-Planck equation \eqref{finitefp-2}
which verifies \eqref{a-priori-estimate-1}.
We shall regard the solutions as functions on $E = H^{-1-\delta}(\T^2)$, i.e. $\rho^N_t(\omega)= \rho^N_t(\Pi_N\omega), (t,\omega) \in \R_+ \times E$. Then, combining \eqref{a-priori-estimate-1} with \eqref{entropy-init} and \eqref{L1-norm}, for a.e. $t >0$,
\begin{equation}\label{entropy-solution}
\int_{E} \rho_t^N \log \rho_t^N \d\mu \leq e^{-2\alpha t}\int_E \rho_0 \log \rho_0\, \d\mu + \big(1- e^{-2\alpha t}\big) \|\rho_0\|_{L^1(\mu)} \log \|\rho_0\|_{L^1(\mu)}.
\end{equation}
From this estimate and a diagonal argument, there exist a subsequence $\big\{\rho^{N_i} \big\}_{i\geq 1}$ and some function $\rho:\R_+ \times E\to \R_+$ such that, for any $T>0$, $\rho^{N_i}$ converges weakly in $L^1\big(0,T; L^1(E,\mu) \big)$ to $\rho$, and for a.e. $t>0$,
$$\int_E \rho_t \log \rho_t\, \d\mu \leq e^{-2\alpha t}\int_E \rho_0 \log \rho_0 \, \d\mu + \big(1- e^{-2\alpha t}\big) \|\rho_0\|_{L^1(\mu)} \log \|\rho_0\|_{L^1(\mu)}.$$
The proof is similar to that of Corollary \ref{cor-1}. Moreover, by the duality of Orlicz spaces, one has, for any $T>0$,
$$\lim_{i\to \infty} \int_0^T \int_E G(t,\omega) \rho_t^{N_i}(\omega)\d\mu\d t = \int_0^T \int_E G(t,\omega) \rho_t(\omega)\d\mu\d t$$
for any $G$ such that, for some small $\eps>0$,
\begin{equation}\label{exp-integrability}
\sup_{t\in [0,T]} \int_E e^{\eps |G(t,\omega)|} \d\mu \d t <+\infty.
\end{equation}

Fix any cylindrical function $\psi$ and $f\in C_c^1(\R_+)$, for $N$ big enough we always have the equation \eqref{finitefp-2}; replacing $N$ by $N_i$, it can be rewritten as
\begin{equation*}
	0 = f(0)\int_{E} \psi \rho_0^{N_i}\d\mu + \int_0^\infty \int_{E}\rho_t^{N_i} \Big[ f^\prime(t) \psi + f(t) \brak{b_{N_i}, D \psi} + \alpha f(t) \L \psi \Big] \d\mu\d t.
\end{equation*}
By the definition \eqref{initial-data-1}, it is not difficult to show that, for any cylindrical $\psi$,
$$\lim_{i\to \infty } \int_{E} \psi \rho_0^{N_i}\d\mu= \int_{E} \psi \rho_0\, \d\mu.$$
Moreover, the first and the third terms in the second integral also converge to the corresponding limits.
The only term that requires our attention is the nonlinear part. We have
\begin{align*}
	\ \bigg| \int_0^\infty \! \int_{E}\rho_t^{N_i} f(t) \brak{b_{N_i}, D \psi} \d\mu\d t \,
	- & \int_0^\infty \! \int_{E}\rho_t f(t) \brak{\mathcal B, D \psi} \d\mu\d t\bigg|\\
	\leq&\, \bigg| \int_0^\infty \! \int_{E}\rho_t^{N_i} f(t) \big(\< b_{N_i}, D \psi\> - \<\mathcal B, D \psi\> \big)\d\mu\d t \bigg|\\
	&\, + \bigg| \int_0^\infty \! \int_{E} \big( \rho_t^{N_i} - \rho_t \big) f(t) \brak{\mathcal B, D \psi} \d\mu\d t \bigg|.
\end{align*}
By \eqref{exponentialint}, $G(t,\omega):= f(t) \brak{\B, D \psi}$ satisfies \eqref{exp-integrability}. Thus, the second term on the right hand side tends to 0 as $i\to \infty$. Next, one can prove that $\brak{b_{N_i}, D \psi}$ converge strongly in $L^1(E,\mu)$ to $\brak{\B, D \psi}$ as $i\to \infty$, see for instance \cite[Section 3.3.1]{DPFlRo17}. Combining the convergence with the uniform exponential integrability of these quantities, we deduce that the sequence $\brak{b_{N_i}, D \psi}$ actually converges to $\brak{\B, D \psi}$ in the Orlicz norm. Therefore, by \eqref{entropy-solution}, the first term also vanishes as  $i\to \infty$. Thus, we can let $i\to \infty$ in the above equality to get the equation
\begin{equation}\label{fpe-bounded-init}
0 = f(0)\int_{E} \psi \rho_0\d\mu + \int_0^\infty \! \int_{E}\rho_t \Big[ f^\prime(t) \psi + f(t) \brak{\B, D \psi} + \alpha f(t) \L \psi \Big] \d\mu\d t.
\end{equation}
Therefore, $\rho_t$ solves the Fokker-Planck equation \eqref{fpe} for $L\log L$ initial condition. The proof of \autoref{thm:llogl} is complete.

\section{\texorpdfstring{$L^p$}{Lp}-initial data}

In this section we assume the initial data of the Fokker-Planck equation \eqref{fpe} to be integrable of order $p>1$.
In this case, we can follow the arguments in the last section to prove the existence of weak solutions to the Fokker-Planck equations \eqref{fpe}. Here we only prove new a priori estimates on the Galerkin approximations and the exponential convergence in $L^2(\mu)$ norm in the case $p=2$.

\subsection{A priori estimates for \texorpdfstring{$p>1$}{p>1}}\label{subsec-Lp}

Assume first $\rho^N_0\in L^\infty(H_N,\mu_N)$ and consider as above the Fokker-Planck equation \eqref{finitefp-1}:
\begin{equation*}
\begin{cases}
\partial_t \rho_t^{(n)}= \big(\A_N^{(n)} \big)^\ast \rho_t^{(n)},\\
\rho^{(n)}|_{t=0} = \rho_0^{(n)}= P^N_{1/n} \rho_0^N.
\end{cases}
\end{equation*}
Jensen's inequality implies
\begin{equation}\label{initial-data-Lp}
\int_{H_N} \big| \rho^{(n)}_0 \big|^p\d\mu_N \leq \int_{H_N} \big|\rho^N_0 \big|^p\d\mu_N \quad \mbox{for all } n\geq 1,
\end{equation}
and we can extend this bound for all subsequent times.

\begin{lemma}\label{lem-Lp}
	For any $n\in\N$, it holds that
	$$\int_{H_N} \big|\rho^{(n)}_t \big|^p\d\mu_N \leq \int_{H_N} \big|\rho^N_0 \big|^p\d\mu_N \quad \mbox{for all } t>0 .$$
\end{lemma}

\begin{proof}
	Using equation \eqref{finitefp-1},
	\begin{align*}
		\partial_t \Big[ \big| \rho^{(n)}_t \big|^p \Big]
        & = p \Big[ \big(\rho^{(n)}_t \big)^2 \Big]^{\frac p2 -1} \rho^{(n)}_t \partial_t \rho^{(n)}_t\\
		&= b_N^{(n)}\cdot \nabla \Big[ \big|\rho^{(n)}_t \big|^p \Big]
		+ p\alpha \Big[ \big(\rho^{(n)}_t \big)^2 \Big]^{ \frac{p-1}2 } \L_N \rho^{(n)}_t.
	\end{align*}
	Integrating by parts on $H_N$ with respect to $\mu_N$ gives us
	\begin{equation*}
		\frac{\d}{\d t} \int_{H_N} \Big[ \big(\rho^{(n)}_t \big)^p \Big] \d\mu_N
		= - p\alpha \int_{H_N} \big(\rho^{(n)}_t \big)^{p-2} \Big|\nabla \rho^{(n)}_t \Big|^2 \d\mu_N.
	\end{equation*}
	Next, integrating in time between $0$ and $t$ leads to
	$$\int_{H_N} \Big[ \big(\rho^{(n)}_t \big)^p \Big] \d\mu_N \leq \int_{H_N} \Big[ \big(\rho^{(n)}_0 \big)^p \Big] \d\mu_N,$$
	which, together with \eqref{initial-data-Lp}, yields the desired estimate.
\end{proof}

As a consequence, $\big\{ \rho^{(n)} \big\}_{n\geq 1}$ is bounded in $L^\infty \big(\R_+, L^p(H_N,\mu_N) \big)$.
Thus we can find a subsequence which converges weakly-$\ast$ to some limit
\begin{equation*}
	\rho^N \in L^\infty \big(\R_+, L^p(H_N,\mu_N) \big),
\end{equation*}
satisfying the estimate
\begin{equation}\label{finitefpe-density}
\sup_{t\in \R_+} \int_{H_N} \big|\rho^N_t \big|^p\d\mu_N \leq \int_{H_N} \big|\rho^N_0 \big|^p\d\mu_N
\end{equation}
and the finite dimensional Fokker-Planck equation
\begin{align}\label{finitefpe-1}
0=&\ f(0)\int_{H_N} \psi \rho_0^N\d\mu_N \\ \nonumber
&\, +\int_0^\infty \! \int_{H_N}\rho_t^N \Big[ f^\prime(t) \psi
+ f(t) \brak{b_N, \nabla \psi}_{H_N} + \alpha f(t) \L_N\psi \Big] \d\mu_N\d t
\end{align}
for any $\psi\in C_b^\infty(H_N)$ and $f\in C_c^1(\R_+ )$.

Next, if $\rho_0\in L^p(E,\mu)$, we define, for $N\in \N$,
\begin{equation}\label{finitefpe-init-data}
\rho^N_0 = P^N_{1/N}\expt{ (-N) \vee (\rho_0\wedge N) \big|\Pi_N},
\end{equation}
which, by Jensen's inequality, satisfies
\begin{equation}\label{finitefpe-init-data-1}
\sup_{N\geq 1} \int_{H_N} \big|\rho^N_0 \big|^p\d\mu_N \leq \int_E |\rho_0|^p \d\mu.
\end{equation}
Consider the finite dimensional Fokker-Planck equations \eqref{finitefpe-1} with initial data $\rho^N_0$,
and regard the solutions $\rho^N_t$ as functions on $E$.
From estimate \eqref{finitefpe-density} and inequality \eqref{finitefpe-init-data-1} we deduce
\begin{equation}\label{finitefpe-density-estim}
\sup_{N\geq 1} \sup_{t\in \R_+ } \int_E \big|\rho^N_t \big|^p\d\mu \leq \int_E |\rho_0|^p \d\mu.
\end{equation}
Hence, we can find a subsequence $\rho^{N_i}$ converging weakly-$\ast$ in $L^\infty \big(\R_+, L^p(E,\mu) \big)$ to some $\rho$,
which can be shown to satisfy the Fokker-Planck equation \eqref{fpe}, thus completing the proof of point (i) of \autoref{thm:lp}.
We omit the details.

\subsection{The case \texorpdfstring{$p=2$}{p=2}}

We want to show the exponential decay of the energy, proving point (ii) of \autoref{thm:lp}.
We start again from equation \eqref{finitefp-1} with the initial condition
$\rho^{(n)}_0 = P^N_{1/n} \rho^N_0$, where $\rho^N_0 \in L^\infty(H_N)$.
It is clear that for all $n\geq 1$,
$$\bar \rho^{(n)}_0 := \int_{H_N} \rho^{(n)}_0 \d\mu_N = \int_{H_N} \rho^N_0 \d\mu_N =:\bar \rho^N_0.$$

\begin{lemma}\label{lem-L2}
	It holds that
	$$\int_{H_N} \big(\rho^{(n)}_t -\bar \rho^N_0\big)^2\d\mu_N \leq e^{-2\alpha t} \int_{H_N} \big(\rho^N_0 -\bar \rho^N_0\big)^2\d\mu_N \quad \mbox{for all } t>0 .$$
\end{lemma}

\begin{proof}
	According to equation \eqref{finitefp-1}, we have
	$$\partial_t \Big[ \big(\rho^{(n)}_t -\bar \rho^N_0\big)^2 \Big] = 2\big(\rho^{(n)}_t -\bar \rho^N_0\big) b_N^{(n)}\cdot \nabla \rho^{(n)}_t + 2\alpha \big(\rho^{(n)}_t -\bar \rho^N_0\big) \L_N \rho^{(n)}_t. $$
	By \eqref{divergence}, integrating by parts with respect to $\mu_N$ yields
	$$\frac{\d}{\d t} \int \big(\rho^{(n)}_t -\bar \rho^N_0\big)^2\d\mu_N = -2\alpha \int \Big| \nabla \rho^{(n)}_t \Big|^2 \d\mu_N. $$
	Recall that $\mu_N$ satisfies the Poincar\'e inequality on $H_N$: for any $\varphi\in W^{1,2}(H_N,\mu_N)$,
	$$\int (\varphi - \bar\varphi)^2\d\mu_N \leq \int |\nabla \varphi|^2\d\mu_N ,$$
	where $\bar\varphi = \int \varphi\, \d\mu_N$. Therefore,
	$$\frac{\d}{\d t} \int \big(\rho^{(n)}_t -\bar \rho^N_0\big)^2\d\mu_N \leq -2\alpha \int \big(\rho^{(n)}_t -\bar \rho^N_0\big)^2 \d\mu_N, $$
	where we used the fact that $\bar \rho^{(n)}_t := \int \rho^{(n)}_t \d\mu_N = \bar \rho^{(n)}_0 = \bar \rho^N_0$ for all $t>0$. As a result,
	$$\int \big(\rho^{(n)}_t -\bar \rho^N_0\big)^2\d\mu_N \leq e^{-2\alpha t} \int \big(\rho^{(n)}_0 -\bar \rho^N_0\big)^2\d\mu_N \quad \mbox{for all } t>0 .$$
	Finally, we complete the proof by noting that
	\begin{align*}
		\int \big(\rho^{(n)}_0 -\bar \rho^N_0\big)^2\d\mu_N &= \int \big(\rho^{(n)}_0 \big)^2\d\mu_N - \big( \bar \rho^N_0 \big)^2\\
		&\leq \int \big(\rho^N_0 \big)^2\d\mu_N - \big( \bar \rho^N_0 \big)^2 = \int \big(\rho^N_0 -\bar \rho^N_0\big)^2\d\mu_N,
	\end{align*}
	where we have used Jensen's inequality in the second step.
\end{proof}

Repeating the arguments below \autoref{lem-Lp}, there exists a subsequence $\rho^{(n_i)}$ converging weakly-$\ast$ to some $\rho^N \in L^\infty\big(\R_+, L^2(H_N,\mu_N)\big)$, which is a weak solution to the finite dimensional Fokker-Planck equations \eqref{finitefpe-1} with the initial datum $\rho^N_0$. Moreover, replacing the set $\mathcal S$ in the proof of \autoref{cor-1} by
  $$\tilde{\mathcal S}= \Big\{ u\in L^2\big([0,T]\times H_N\big): \big\|u_t- \bar \rho^N_0 \big\|_{L^2(\mu_N)} \leq e^{-\alpha t} \big\|\rho^N_0- \bar \rho^N_0 \big\|_{L^2(\mu_N)} \ \forall\, t\in [0,T] \Big\},$$
similar discussions imply that for a.e. $t\in (0,T)$, one has
  \begin{equation*}
  \big\|\rho^N_t- \bar \rho^N_0 \big\|_{L^2(\mu_N)} \leq e^{-\alpha t} \big\|\rho^N_0- \bar \rho^N_0 \big\|_{L^2(\mu_N)} .
  \end{equation*}
The arbitrariness of $T>0$ yields that the above inequality holds for a.e. $t>0$.

Next, for $\rho_0\in L^2(E,\mu)$ and $N\in \N$, we define $\rho^N_0$ as in \eqref{finitefpe-init-data}. We have
$$\bar \rho^N_0 = \int_{H_N} \rho^N_0 \d\mu_N = \int_{H_N} \E\big[ (-N) \vee (\rho_0\wedge N) \big| \Pi_N \big] \d\mu_N = \int_E (-N) \vee (\rho_0\wedge N) \d\mu, $$
therefore,
$$  \lim_{N\to \infty} \bar \rho^N_0 = \int_E \rho_0\, \d\mu = \bar \rho_0. $$
This together with \eqref{finitefpe-init-data-1} (taking $p=2$) implies
\begin{equation}\label{limit-mean}
\limsup_{N\to \infty} \int_{H_N} \big(\rho^N_0 -\bar \rho^N_0\big)^2\d\mu_N \leq \int_E (\rho_0 -\bar\rho_0)^2 \d\mu.
\end{equation}
For any $N\geq 1$, there exists a weak solution $\big(\rho^N_t \big)_{t\in \R_+}$ to the equation \eqref{finitefpe-1} with the initial condition $\rho^N_0$, satisfying
  \begin{equation}\label{expon-decay-energy}
  \big\|\rho^N_t- \bar \rho^N_0 \big\|_{L^2(\mu_N)} \leq e^{-\alpha t} \big\|\rho^N_0- \bar \rho^N_0 \big\|_{L^2(\mu_N)} \quad \mbox{for a.e. } t\in (0,\infty).
  \end{equation}
As usual, we view $\rho^N_t  (N\geq 1)$ as functionals on $E$. As in Section \ref{subsec-Lp}, there is a subsequence $\rho^{N_i}$ converging weakly-$\ast$ to some $\rho \in L^\infty \big(\R_+, L^2(E,\mu)\big)$.  By \eqref{limit-mean} and \eqref{expon-decay-energy}, we can show the exponential decay of the energy of $\rho_t$ for a.e. $t>0$.

%
\section{Existence of Weak solutions}

Thanks to the control on densities we have gained in the last Section, we are now in the position to prove \autoref{thm:flow}.
Let us thus take $\rho_0 \in L^p(E,\mu; \R_+)$ for some $p>1$, satisfying $\bar\rho_0= \int_E \rho_0\, \d\mu =1$.
We define $\rho^N_0$ similarly to \eqref{finitefpe-init-data}:
  \begin{equation}\label{init-densities}
  \rho^N_0 = c_N^{-1} P^N_{1/N} \expt{(\rho_0\wedge N) \big|\Pi_N},
  \end{equation}
where $c_N $ is the normalizing constant such that $\bar\rho^N_0 = \int_{H_N} \rho^N_0 \d\mu_N =1$. Clearly,
  $$\lim_{N\to \infty} c_N =1.$$
Let $\rho^N_t$ be the solution of the finite dimensional Fokker-Planck equations \eqref{finitefpe-1} with initial data $\rho^N_0$. Combining the above fact with \eqref{finitefpe-density-estim}, we see that
  \begin{equation}\label{uniform-density-estim}
  \sup_{N\geq 1} \sup_{t\in [0,T]} \big\|\rho^N_t \big\|_{L^p(\mu)} \leq c_0 \|\rho_0\|_{L^p(\mu)} .
  \end{equation}

Consider the solution $\omega^N_t$ of the SDEs \eqref{sde}, for which the initial values $\omega^N_0$ is distributed as $\rho^N_0\mu_N$; then  $\rho^N_t$ is the probability density function (with respect to $\mu_N$) of $\omega^N_t$. In this part we regard $\omega^N_t$ and $\rho^N_t$ as objects defined on $E= H^{-1-}$, i.e. $\omega^N_t(\omega) = \omega^N_t(\Pi_N \omega)$, $\rho^N_t(\omega) = \rho^N_t(\Pi_N \omega)$. We want to show that the laws $Q^N$ of $\omega^N_\cdot$ on $C\big( [0,T], E \big)$ are tight. To this end we will use the compactness criterion proved in \cite[Corollary 9, p. 90]{Si87}. The arguments here follow those of \cite[Section 3]{FlLu17}.

Take $\delta\in (0,1)$, $\kappa>5$ (this choice is due to estimates below) and consider the spaces
  $$X=H^{-1-\delta/2}(\T^2),\quad B=H^{-1-\delta}(\T^2),\quad Y=H^{-\kappa}(\T^2).$$
Then $X\subset B\subset Y$ with compact embeddings and we also have, for a suitable constant $C>0$ and for
  \begin{equation}\label{eq-theta}
  \theta= \frac{\delta/2}{\kappa -1-\delta/2},
  \end{equation}
the interpolation inequality
  $$\|\omega\|_B \leq C \|\omega\|_X^{1-\theta} \|\omega\|_Y^\theta,\quad \omega\in X.$$
These are the preliminary assumptions of \cite[Corollary 9, p. 90]{Si87}. We consider here a particular case:
  $$\mathcal S= L^{p_0}(0,T; X)\cap W^{1/3,4}(0,T; Y),$$
where for $0< \alpha <1$ and $p\geq 1$,
  $$W^{\alpha,p}(0,T; Y)=\bigg\{f:  f\in L^p(0,T; Y) \mbox{ and } \int_0^T \int_0^T \frac{\|f(t)-f(s)\|_Y^p}{|t-s|^{\alpha p+1}}\d t\d s <\infty\bigg\}.$$

\begin{lemma}\label{lem-embedding}
Let $\delta\in (0,1)$ and $\kappa>5$ be given. If
  $$p_0> \frac{12(\kappa -1-3\delta/2)}\delta,$$
then $\mathcal S$ is compactly embedded into $C\big([0,T], H^{-1-\delta}(\T^2) \big)$.
\end{lemma}

\begin{proof}
Recall that $\theta$ is defined in \eqref{eq-theta}. In our case, we have $s_0=0, r_0=p_0$ and $s_1=1/3, r_1=4$. Hence $s_\theta = (1-\theta)s_0 +\theta s_1= \theta/3$ and
  $$\frac1{r_\theta} = \frac{1-\theta}{r_0} + \frac\theta{r_1} = \frac{1-\theta}{p_0} + \frac\theta 4.$$
It is clear that for $p_0$ given above, it holds $s_\theta> 1/r_\theta$, thus the desired result follows from the second assertion of \cite[Corollary 9]{Si87}.
\end{proof}

For $N\geq 1$, let $Q^N$ be the law of $\omega^N_\cdot$ on $\mathcal X:= C\big([0,T], H^{-1-}(\T^2) \big)$. We want to prove that the family $\big\{Q^N\big\}_{N\geq 1}$ is tight in $\mathcal X$. The next result follows from the definition of the topology in $\mathcal X$.

\begin{lemma}\label{lem-tight}
The family $\big\{Q^N\big\}_{N\geq 1}$ is tight in $\mathcal X$ if and only if it is tight in
the space $C\big([0,T], H^{-1-\delta}(\T^2) \big)$ for any $\delta>0$.
\end{lemma}

In view of the above two lemmas, it is sufficient to prove that  $\big\{Q^N\big\}_{N\geq 1}$ is bounded in probability in $W^{1/3,4} \big(0,T; H^{-\kappa}(\T^2) \big)$ and in each $L^{p_0}\big(0,T; H^{-1-\delta}(\T^2) \big)$ for any $p_0>0$ and $\delta>0$.

We show first that the family $\big\{Q^N\big\}_{N\geq 1}$ is bounded in probability on the space $L^{p_0}\big(0,T; H^{-1-\delta}(\T^2) \big)$. Let us recall that, for any $q>1$ and $\delta>0$, there exists $C_{q,\delta}>0$ such that
  $$\int \|\omega\|_{H^{-1-\delta}}^q \, \d\mu \leq C_{q,\delta}.$$
We have
\begin{equation}\label{sec-3.1}
    \aligned
	\expt{\int_0^T \big\|\omega^N_t \big\|_{H^{-1-\delta}}^{p_0} \d t}
	&=\int_0^T \expt{\big\|\omega^N_t \big\|_{H^{-1-\delta}}^{p_0}}\d t\\
	&=\int_0^T  \int \|\omega \|_{H^{-1-\delta}}^{p_0} \rho^N_t(\omega)\d\mu \d t\\
	& \leq \int_0^T \bra{\int \|\omega \|_{H^{-1-\delta}}^{p_0 q}\d\mu}^{1/q} \bra{\int \big(\rho^N_t(\omega) \big)^p \d\mu}^{1/p}\d t \\
	& \leq  C_{p_0 q, \delta} T \sup_{t\in [0,T]} \big\| \rho^N_t\big\|_{L^p(\mu)} \leq C_{p_0q, \delta}  T \| \rho_0 \|_{L^p(\mu)},
\endaligned
\end{equation}
where $q$ is the conjugate number of $p$ and we have used the above estimate and \eqref{uniform-density-estim} in the last two steps.
By Chebyshev's inequality, the family $\big\{Q^N\big\}_{N\geq 1}$ is bounded in probability in $L^{p_0}\big(0,T; H^{-1-\delta}(\T^2) \big)$.

Next, we prove boundedness in probability of $\big\{Q^N\big\}_{N\geq 1}$ in $W^{1/3,4}\big(0,T; H^{-\kappa}(\T^2) \big)$ where $\kappa >5$. Again by Chebyshev's inequality, it suffices to show that
  $$\sup_{N\geq 1} \E \bigg[\int_0^T \big\|\omega^N_t \big\|_{H^{-\kappa}}^4 \d t +\int_0^T \int_0^T \frac{\big\| \omega^N_t- \omega^N_s \big\|_{H^{-\kappa}}^4} {|t-s|^{7/3}}\d t\d s\bigg] <\infty. $$
In view of \eqref{sec-3.1}, we see that it is sufficient to establish a uniform estimate on the expectation $\E \big\|\omega^N_t- \omega^N_s \big\|_{H^{-\kappa}}^4$. We write $\<\cdot, \cdot\>$ for the inner product in $L^2(\T^2)$.

\begin{lemma}\label{lem-estimate}
There exists $C>0$ depending on $\alpha, \delta$ and $\| \rho_0\|_{L^p(\mu)}$ such that for any $k\in \Lambda_N$, we have
  $$\E\big[ \big\<\omega^N_t- \omega^N_s, e_k \big\>^4 \big]\leq C (t-s)^2\big( |k|^8+1 \big).$$
\end{lemma}

\begin{proof}
By equation \eqref{sde},
  $$\aligned
 	\big\<\omega^N_t, e_k\big\>
 	&= \big\<\omega^N_0, e_k\big\> + \int_0^t \big\<\omega^N_s, u\big(\omega^N_s\big) \cdot \nabla e_k\big\> \d s \\
 	&\quad - \alpha \int_0^t \big\<\omega^N_s, e_k\big\> \d s + \sqrt{2\alpha} \int_0^t \big\<\d W^{(N)}_s, e_k\big\> \\
 	&= \big\<\omega^N_0, e_k\big\> + \int_0^t \big\<\omega^N_s \otimes \omega^N_s, H_{e_k} \big\> \d s - \alpha \int_0^t \big\<\omega^N_s, e_k\big\> \d s + \sqrt{2\alpha}\, W^k_t.
 \endaligned $$
Therefore, for $0\leq s<t\leq T$,
 \begin{equation}\label{lem-estimate-1}
 \big\<\omega^N_t- \omega^N_s, e_k \big\> = \int_s^t \big\<\omega^N_r\otimes \omega^N_r, H_{e_k} \big\> \d r -\alpha \int_s^t \big\<\omega^N_r, e_k\big\> \d r + \sqrt{2\alpha} (W^k_t- W^k_s).
 \end{equation}
First, we control by H\"older's inequality:
 \begin{align*}
 	&\ \E\bigg[\bigg(\int_s^t \big\<\omega^N_r\otimes \omega^N_r, H_{e_k}\big\> \d r\bigg)^{ 4} \bigg] \\
    \leq&\ (t-s)^3\, \E\bigg[\int_s^t \big\<\omega^N_r\otimes \omega^N_r, H_{e_k} \big\>^4 \d r\bigg]\\
 	=&\ (t-s)^3 \int_s^t \int \big\<\omega\otimes \omega, H_{e_k} \big\>^4 \rho^N_r \d\mu\d r\\
 	\leq&\ (t-s)^3 \int_s^t \bigg[ \int \big\<\omega\otimes \omega, H_{e_k} \big\>^{4q}\d \mu \bigg]^{1/q}
 	\bigg[ \int \big(\rho^N_r \big)^p \d \mu \bigg]^{1/p} \d r.
 \end{align*}
By \eqref{goodbound} and the uniform density estimate \eqref{uniform-density-estim},
  \begin{equation}\label{lem-estimate-2}
  \aligned
  \E\bigg[\bigg(\int_s^t \big\<\omega^N_r\otimes \omega^N_r, H_{e_k}\big\> \d r\bigg)^{ 4} \bigg] &\leq C_q \|e_k\|_{C^2(\T^2)}^4 (t-s)^4 \sup_{t\in [0,T]} \big\| \rho^N_r\big\|_{L^p(\mu)}\\
  & \leq C_q (t-s)^4  |k|^8 \| \rho_0\|_{L^p(\mu)}.
  \endaligned
  \end{equation}
Similarly,
  \begin{equation}\label{lem-estimate-3}
  \aligned
  \E \bigg[ \bigg(\int_s^t \big\<\omega^N_r, e_k\big\> \d r \bigg)^4 \bigg] &\leq (t-s)^3\, \E \int_s^t \big\<\omega^N_r, e_k\big\>^4 \d r\\
  &= (t-s)^3 \int_s^t \int \<\omega, e_k \>^4 \rho^N_r \d\mu \d r \\
  &\leq C_q (t-s)^4  \| \rho_0\|_{L^p(\mu)}.
  \endaligned
  \end{equation}
Finally,
  $$\E \big[ (W^k_t- W^k_s)^4 \big] \leq C(t-s)^2. $$
Combining this estimate with \eqref{lem-estimate-1}--\eqref{lem-estimate-3} yields the result.
\end{proof}

As a result of Lemma \ref{lem-estimate}, by Cauchy's inequality,
  $$\aligned
  \E \big( \big\|\omega^N_t- \omega^N_s \big\|_{H^{-\kappa}}^4 \big) &= \E\Bigg[\bigg( \sum_{k\in \Z_0^2} |k|^{-2\kappa} \big\<\omega^N_t- \omega^N_s, e_k \big\>^2 \bigg)^{ 2} \Bigg]\\
  &\leq \bigg(\sum_{k\in \Z_0^2} |k|^{-2\kappa} \bigg) \sum_{k\in \Z_0^2} |k|^{-2\kappa}\, \E \Big[ \big\<\omega^N_t- \omega^N_s, e_k \big\>^4 \Big]\\
  &\leq \tilde C (t-s)^2\sum_{k\in \Z_0^2} |k|^{-2\kappa}  |k|^8 \leq \hat C (t-s)^2,
  \endaligned$$
since $2\kappa -8 >2$ due to the choice of $\kappa$. Consequently,
  $$\E \bigg[\int_0^T \int_0^T \frac{ \big\|\omega^N_t- \omega^N_s \big\|_{H^{-\kappa}}^4} {|t-s|^{7/3}}\d t\d s\bigg] \leq \hat C \int_0^T \int_0^T \frac{|t-s|^2} {|t-s|^{7/3}}\d t\d s <\infty.$$
The proof of the boundedness in probability of $\big\{Q^N\big\}_{N\geq 1}$ in $W^{1/3,4}\big(0,T; H^{-\kappa}(\T^2) \big)$ is complete.

To summarize, we have shown that the family $\big\{Q^N \big\}_{N\geq 1}$ of laws of $\big\{ \omega^N_\cdot \big\}_{N\geq 1}$ is tight on $\mathcal X=C([0,T], E )$. Since we are dealing with the SDEs \eqref{sde}, it is necessary to consider the laws of $\omega^N_\cdot$ together with the law $\mathcal W$ on $\mathcal Y=C\big([0,T], \R^{\Z^2_0} \big)$ of the family of Brownian motions $W:= \{W^k_\cdot \}_{k\in \Z^2_0}$. For any $N\in \N$, we denote $Q^N \otimes \mathcal W$ the joint law (not the product measure) of $(\omega^N_\cdot, W)$ on
  $$\mathcal X\times \mathcal Y= C([0,T], E ) \times C\Big([0,T], \R^{\Z^2_0} \Big).$$
Then, it is easy to see that the family $\big\{ Q^N \otimes \mathcal W\big\}_{N\geq 1}$ of joint laws is tight on $\mathcal X\times \mathcal Y$, cf. the arguments above  \cite[Lemma 3.4]{FlLu17}. Thus, by Prohorov's theorem (see \cite[Theorem 5.1, p. 59]{Bi99}), we can find a subsequence $\big\{Q^{N_i} \otimes \mathcal W \big\}_{i\geq 1}$ which converge weakly to some $Q\otimes \mathcal W$, a probability measure on $\mathcal X\times \mathcal Y$. Next, the Skorokhod theorem (see \cite[Theorem 6.7, p. 70]{Bi99} implies that there exist a probability space $\big(\tilde \Theta, \tilde{\mathcal F}, \tilde\P \big)$, a sequence of  processes $\big\{ \big(\tilde \omega^{N_i}_\cdot, \tilde W^{N_i} \big) \big\}_{i\in \N}$ and a limit process $\big(\tilde\omega_\cdot, \tilde W \big)$ defined on this probability space such that, for all $i\in \N$, the law of $\big(\tilde \omega^{N_i}_\cdot, \tilde W^{N_i} \big)$ is $Q^{N_i} \otimes \mathcal W$, and $\tilde\P$-a.s., $\big(\tilde \omega^{N_i}_\cdot, \tilde W^{N_i} \big)$ converge in $\mathcal X\times \mathcal Y$ to $\big(\tilde\omega_\cdot, \tilde W \big)$ as $i\to \infty$. Note that $\tilde W^{N_i}$ and $\tilde W$ are families of Brownian motions indexed by $\Z^2_0$.

We need one last result before proving the existence of solutions to \eqref{ddeuler}.

\begin{lemma}\label{lem-limit-density}
For a.e. $t\in [0,T]$, the law of $\tilde\omega_t$ on $E$ has a density $\rho_t$ with respect to $\mu$, where $\rho_t$ is a weak solution to the Fokker--Planck equation \eqref{fpe}.
\end{lemma}

\begin{proof}
Fix any $F\in C_b(E,\R)$ and $f\in C([0,T])$. By the $\tilde \P$-a.s. convergence of $\tilde \omega^{N_i}_\cdot$ to $\tilde \omega_\cdot$ in $\mathcal X = C([0,T], E)$, we have
  $$\aligned
  \E_{\tilde \P} \int_0^T f(t) F(\tilde \omega_t)\d t &= \lim_{i\to \infty} \E_{\tilde \P}\int_0^T f(t) F\big(\tilde \omega^{N_i}_t \big) \d t= \lim_{i\to \infty} \E_{ \P} \int_0^T f(t) F\big(\omega^{N_i}_t \big)\d t \\
  &= \lim_{i\to \infty} \int_0^T f(t) \int_E F(\omega) \rho^{N_i}_t(\omega) \,\d\mu(\omega) \d t.
  \endaligned $$
The densities $\rho^{N_i}_\cdot  (i\in \N)$ satisfy the estimates \eqref{uniform-density-estim}, thus, taking a further subsequence if necessary, we can assume that $\rho^{N_i}_\cdot$ converge weakly to some limit $\rho_\cdot$, which by the first half of Theorem \ref{thm:lp}, is a weak solution of the Fokker--Planck equation \eqref{fpe}. Next, we have
  $$ \int_0^T f(t)  \E_{\tilde \P} F(\tilde \omega_t)\, \d t = \int_0^T f(t) \int_E F(\omega) \rho_t(\omega)\, \d\mu(\omega) \d t. $$
The arbitrariness of $f\in C([0,T])$ implies that, for a.e. $t\in [0,T]$,
  $$\E_{\tilde \P} F(\tilde \omega_t) = \int_E F(\omega) \rho_t(\omega) \,\d\mu(\omega) .$$
We can take a countable dense subset $\mathcal C\subset C_b(E,\R)$ of functionals $F$ such that, for a.e. $t\in [0,T]$, the above equality holds for all $F\in \mathcal C$. Thus the law of $\tilde \omega_t$ is $\rho_t$.
\end{proof}

Up to now, we have indeed obtained the assertions (i) and (ii) of \autoref{thm:flow}. Finally, we can prove the existence of weak solutions to the stochastic Euler equation \eqref{ddeuler}.

\begin{proof}[Proof of \autoref{thm:flow}(iii)]
Recall that $\omega^{N_i}$ solves the finite dimensional equation \eqref{sde} with $N_i$ in place of $N$, and $\big( \tilde\omega^{N_i}, \tilde W^{N_i} \big)$ has the same law as $\big( \omega^{N_i}, W \big)$, where we write $W$ for the family of Brownian motions $\{W^k_\cdot \}_{k\in \Z^2_0}$, similarly for $\tilde W^{N_i}$. Therefore, for any $\phi \in C^\infty (\T^2)$,
  $$\big\< \tilde\omega^{N_i}_t, \phi \big\> = \big\< \tilde\omega^{N_i}_0, \phi \big\> + \int_0^t \big\< \tilde\omega^{N_i}_s, \big( K\ast \tilde\omega^{N_i}_s\big) \cdot \nabla \phi_{N_i}\big\> \d s - \alpha \int_0^t \big\< \tilde\omega^{N_i}_s, \phi \big\> \d s + \sqrt{2\alpha} \big\< \tilde W^{N_i}_t, \phi \big\>,$$
where $\phi_{N_i}= \Pi_{N_i} \phi = \sum_{k\in \Lambda_{N_i}} \<\phi, e_k\> e_k$. In this equation, we write $\<\cdot, \cdot\>$ for the inner product in $L^2(\T^2)$, which will also be used for the pairing between the distributions $C^\infty(\T^2)'$ and smooth functions $C^\infty(\T^2)$.

By the $\tilde\P$-a.s. convergence of $\big(\tilde \omega^{N_i}_\cdot, \tilde W^{N_i} \big)$ to $\big(\tilde\omega_\cdot, \tilde W \big)$ in $\mathcal X\times \mathcal Y$ as $i\to \infty$, it is clear that all the terms, except the nonlinear part, converge in $L^1\big(\tilde\Theta, \tilde\P, C([0,T],\R) \big)$ to the corresponding one in the limit. Next,
  $$ \int_0^t \big\< \tilde\omega^{N_i}_s, \big( K\ast \tilde\omega^{N_i}_s\big) \cdot \nabla \phi_{N_i}\big\> \d s =  \int_0^t \Big\< \tilde\omega^{N_i}_s \otimes \tilde\omega^{N_i}_s, H_{ \phi_{N_i}} \Big\> \d s.$$
We have
  $$\aligned
  &\ \E_{\tilde \P} \bigg[1\wedge \sup_{0\leq t\leq T} \bigg| \int_0^t \Big\< \tilde\omega^{N_i}_s \otimes \tilde\omega^{N_i}_s, H_{ \phi_{N_i}} \Big\> \d s - \int_0^t \big\< \tilde\omega_s \otimes \tilde\omega_s, H_{ \phi} \big\> \d s\bigg| \bigg] \\
  \leq &\ \E_{\tilde \P} \bigg[1\wedge \int_0^T \Big|\Big\< \tilde\omega^{N_i}_s \otimes \tilde\omega^{N_i}_s, H_{ \phi_{N_i}} \Big\> - \big\< \tilde\omega_s \otimes \tilde\omega_s, H_{ \phi} \big\> \Big| \d s \bigg] \\
  \leq &\ \E_{\tilde \P} \bigg[1\wedge \int_0^T \Big|\Big\< \tilde\omega^{N_i}_s \otimes \tilde\omega^{N_i}_s, H_{ \phi_{N_i}} -H_{ \phi} \Big\> \Big| \d s \bigg] \\
  &\, + \E_{\tilde \P} \bigg[1\wedge \int_0^T \Big|\Big\< \tilde\omega^{N_i}_s \otimes \tilde\omega^{N_i}_s - \tilde\omega_s \otimes \tilde\omega_s, H_{ \phi} \Big\> \Big| \d s\bigg] .
  \endaligned$$
We denote the two terms on the right hand side by $I^{N_i}_1$ and $I^{N_i}_2$, respectively. By the definition of $H_\phi$, we have $H_{ \phi_{N_i}} -H_{ \phi}= H_{\phi_{N_i} -\phi}$. Therefore, by \eqref{goodbound},
  $$\aligned
  I^{N_i}_1 & \leq \E_{\tilde \P} \int_0^T \Big|\Big\< \tilde\omega^{N_i}_s \otimes \tilde\omega^{N_i}_s, H_{ \phi_{N_i}} -H_{ \phi} \Big\> \Big| \d s \\
  &\leq C T \|\phi_{N_i}- \phi \|_{C^2(\T^2)} \sup_{0\leq s\leq T} \big\|\rho^{N_i}_s \big\|_{L^p(\mu)} \leq C'T \|\rho_0 \|_{L^p(\mu)} \|\phi_{N_i}- \phi \|_{C^2(\T^2)},
  \endaligned$$
where the last step follows from \eqref{uniform-density-estim}. Since $\phi\in C^\infty(\T^2)$, the Fourier series $\phi_N =\Pi_N \phi$ converge to $\phi$ in $C^\infty(\T^2)$. Thus we deduce
  \begin{equation}\label{proof-0}
  \lim_{i\to \infty} I^{N_i}_1=0.
  \end{equation}

Next, let $H^n_\phi \in C^\infty (\T^2\times \T^2)$ be an approximating sequence of $H_\phi$ as in \eqref{eq:approxH} and \eqref{eq:approxH-1}. By the triangle inequality,
  \begin{equation}\label{proof-1}
  \aligned
  I^{N_i}_2\leq &\ \E_{\tilde \P} \bigg[1\wedge \int_0^T \Big|\Big\< \tilde\omega^{N_i}_s \otimes \tilde\omega^{N_i}_s , H^n_{ \phi} - H_\phi \Big\> \Big| \d s \bigg] \\
  &\, + \E_{\tilde \P} \bigg[1\wedge \int_0^T \big|\big\< \tilde\omega_s \otimes \tilde\omega_s, H^n_{ \phi} -H_\phi \big\> \big| \d s \bigg] \\
  &\, + \E_{\tilde \P} \bigg[1\wedge \int_0^T \Big|\Big\< \tilde\omega^{N_i}_s \otimes \tilde\omega^{N_i}_s - \tilde\omega_s \otimes \tilde\omega_s, H^n_{ \phi} \Big\> \Big| \d s \bigg]\\
  =: & \ J^{N_i}_{1,n} + J_{2,n} + J^{N_i}_{3,n}.
  \endaligned
  \end{equation}
Recall that, by \autoref{lem-limit-density}, $\tilde\omega_s$ has the density $\rho_s$ for a.e. $s\in (0,T)$ and the estimate below holds:
  $$\sup_{0\leq s\leq T} \| \rho_s \|_{L^p(\mu)} \leq \liminf_{i\to \infty} \sup_{0\leq s\leq T} \big\| \rho^{N_i}_s \big\|_{L^p(\mu)} \leq c_0 \| \rho_0 \|_{L^p(\mu)}. $$
Therefore, by \eqref{goodbound.0},
  $$\aligned
  J_{2,n} &\leq \E_{\tilde \P} \int_0^T \big|\big\< \tilde\omega_s \otimes \tilde\omega_s, H^n_{ \phi} -H_\phi \big\> \big| \d s  \\
  &\leq T \bigg[ C_p \big\|  H^n_{ \phi} -H_\phi \big\|_{L^2(\T^2 \times \T^2)}^{1/p'} + \bigg| \int_{\T^2} H^n_\phi(x,x)\,\d x \bigg| \bigg]
  \endaligned$$
which tends to 0 as $n\to \infty$. Next, thanks to the uniform estimates \eqref{uniform-density-estim} on the densities $\rho^{N_i}_s$ of $\tilde\omega^{N_i}_s$, the same arguments as above yield
  $$\lim_{n\to \infty} J^{N_i}_{1,n} =0 \quad \mbox{uniformly in } i\in \N. $$

Finally, fix any $n\in \N$; $\tilde\P$-a.s., $\tilde\omega^{N_i}_\cdot$ converge in $C([0,T], E)$ to $\tilde\omega_\cdot$ as $i\to \infty$, thus
  $$\lim_{i\to \infty} \int_0^T \Big|\Big\< \tilde\omega^{N_i}_s \otimes \tilde\omega^{N_i}_s - \tilde\omega_s \otimes \tilde\omega_s, H^n_{ \phi} \Big\> \Big| \d s=0. $$
As a result, for any fixed $n$, the dominated convergence theorem implies
  $$\lim_{i\to \infty} J^{N_i}_{3,n} =0. $$
Therefore, first letting $i\to \infty$ and then $n\to \infty$ in \eqref{proof-1}, we obtain
  $$\lim_{i\to \infty} I^{N_i}_2=0. $$
Combining this limit with \eqref{proof-0} we finish the proof.
\end{proof}

\section{Gibbsian Energy-Enstrophy Measures}\label{sec:energyens}

We conclude our study with a relevant example of an absolutely continuous measure with respect to the white noise measure $\mu$
from which to start the stochastic dynamics we have discussed so far.

The enstrophy measure $\mu$ is formally represented by
\begin{equation*}
\mu(\d \omega) =\frac1Z e^{-S(\omega)} \d\omega,
\end{equation*}
where $S(\omega)=\frac12\int_{\T^2}\omega^2\,\d x$ is the enstrophy of $\omega$.
As already mentioned in \autoref{sec:mainresult}, $\mu$ is interpreted as the law of the centred, zero averaged (recall that all our function spaces are subject to the zero space average
condition), Gaussian random field $\eta$ with identity covariance kernel, or equivalently the $L^2(\T^2)$ inner product as covariance quadratic form.

Besides enstrophy, 2-dimensional Euler equation preserves \emph{energy},
\begin{equation}\label{eq:energy}
E(\omega)=\frac12 \int_{\T^2} |u|^2 \d x=-\frac12 \int_{\T^2} \omega \Delta^{-1}\omega\, \d x,
\end{equation}
where the second expression is readily obtained from the first one recalling that $u=\nabla^\perp \Delta^{-1}\omega$
and integrating by parts. With this in mind, it is natural to consider another candidate invariant measure for Euler equation,
the \emph{energy-enstrophy measure}
\begin{equation}
\mu_\beta(\d \omega)=\frac1{Z_\beta} e^{-\beta E(\omega)-S(\omega)}\d\omega,
\end{equation}
with $\beta\in\R$ a real parameter. The measure $\mu_\beta$ is rigorously defined as the law of the centred, zero averaged,
Gaussian random field $\eta_\beta$ on $\T^2$ with covariance
\begin{equation*}
\forall f,g\in L^2(\T^2), \quad 	
\expt{\brak{\eta_\beta,f}\brak{\eta_\beta,g}}=\brak{f,Q_\beta g},
\quad Q_{\beta}=(1+\beta(-\Delta)^{-1})^{-1},
\end{equation*}
whenever $Q_\beta$ is well-defined as a positive definite operator, that is for $\beta>-1$.
Equivalently, $\eta_\beta$ is a centred Gaussian stochastic process indexed by $L^2(\T^2)$ with the
specified covariance. Since the embedding of $Q_{\beta}^{1/2}L^2(\T^2)$ into $H^s(\T^2)$ is Hilbert-Schmidt for all $s<-1$,
$\eta_{\beta}$ can be identified with a random distribution taking values in the latter spaces (see \cite{DPZa14}).

The Gaussian random distributions we just introduced are best understood in terms of Fourier series: we can write
\begin{equation*}
\eta_{\beta}=\sum_{k\in\Z^2_0} \hat\eta_{\beta,k} e_k,
\quad \text{ where }\quad
\hat\eta_{\beta,k}=\brak{\eta_\beta,e_k}\sim N_\C\pa{0,\frac{|k|^2}{\beta+ |k|^2}}
\end{equation*}
are independent $\C$-valued Gaussian variables, and the Fourier expansion thus converges in $L^2\pa{H^s(\T^2),\mu_{\beta}}$
for $s<-1$.
The measure $\mu_{\beta}$ is also characterized by its Fourier transform (characteristic function) on $H^s(\T^2)$:
for any $f\in H^{-s}(\T^2)$,
\begin{equation}\label{gaussianchartorus}
\int e^{i\brak{\omega,f}} \d\mu_{\beta}(\omega)=\exp\pa{-\frac{1}{2}\sum_{k\in\Z^2_0}
	\frac{4\pi^2 |k|^2 |\hat f_k|^2}{\beta+4\pi^2 |k|^2}}.
\end{equation}

Let us give an equivalent definition of $\mu_{\beta}$: in sight of \eqref{eq:energy}, energy can be written in terms of Fourier components as
\begin{equation*}
2E(\omega)=-\brak{\omega,\Delta^{-1}\omega}=\sum_{k\in\Z^2_0} \frac{|\hat \omega_k|^2}{|k|^2}.
\end{equation*}
This expression \emph{does not make sense as a random variable} if $\omega=\eta$,
since in that case $\hat\eta_k$'s are i.i.d. Gaussian variables, and the series diverges almost surely.
However, one can define a \emph{renormalized energy} by means of Wick ordering:
\begin{equation}\label{normenergy}
2\wick{E}(\eta) =\lim_{K\rightarrow\infty}  \sum_{|k|\leq K} \frac{\wick{\hat \eta_k\hat\eta_k^\ast}}{|k|^2}=
\lim_{K\rightarrow\infty} \sum_{|k|\leq K} \pa{\frac{|\hat \eta_k|^2}{|k|^2}
	-\int \frac{|\hat \eta_k|^2}{|k|^2} \d\mu(\eta)},
\end{equation}
where the limit holds in $L^2(\mu)$ (see \cite{AlRFHK79}), and it defines an element of
the second Wiener chaos $H^{\wick{2}}(\mu)$. As a consequence, $\wick{E}$ can be expressed
as a double It\^{o}-Wiener stochastic integral with respect to the white noise $\eta$,
the kernel being naturally Green's function $G$:
\begin{equation*}
2\wick{E}(\eta)=I^2(G,\eta),
\end{equation*}
where the right-hand side is the double It\^{o}-Wiener integral of $G(x,y)$ with respect to the white noise $\eta$
(that is, on the Gaussian space $(E,\mu)$), as defined in \cite{Ja97}.
The proof of the forthcoming Proposition is detailed in \cite{Gr19}, and has an analogue in infinite product representations
of energy-enstrophy measures in \cite{AlCr90}.

\begin{proposition}
	The probability measure on $E$ defined by
	\begin{equation}\label{normenergymeasure}
	\d\tilde\mu_\beta=\frac{1}{Z_\beta}e^{-\beta\wick{E}(\omega)}\d\mu(\omega),
	\quad Z_\beta=\int e^{-\beta\wick{E}(\omega)}\d\mu(\omega),
	\end{equation}
	is well-posed. It coincides with the energy-enstrophy measure, $\tilde\mu_{\beta}=\mu_{\beta}$.
\end{proposition}

Intuition suggests that the renormalized energy is invariant for Euler's equation, and we can express this fact rigorously
by means of the above discussion. The idea is to exhibit a solution of the Fokker-Planck equation \eqref{fpe}
--- in the case where friction and forcing are absent, $\alpha=0$ --- such that $\rho_t\equiv\  \wick{E}$.
In fact, since no uniqueness results are available, this is the best notion of invariance we can produce,
and as we see below it is a consequence of the infinitesimal invariance already observed in the literature.

\begin{proposition}
	For any cylinder function $\phi\in\FC_b$ and $\beta>-1$ it holds
	\begin{equation}\label{infinitesimal}
	\expt{\wick{E}(\eta)\, \B \phi(\eta)}=\expt{\frac{1}{Z_\beta}e^{-\beta\wick{E}(\eta)}\B \phi(\eta)}=0.
	\end{equation}
	As a consequence, for $\alpha=0$, there exist constant solutions of \eqref{fpe}
	(in the sense specified in \autoref{sec:mainresult})
	such that $\rho_t\equiv \wick{E}$ or $\frac{1}{Z_\beta}e^{-\beta\wick{E}}$.
	Moreover, there exists a weak solution of \eqref{stocheuler} (again in the sense of \autoref{sec:mainresult})
	whose fixed time marginals are constant in time, and coincide
	with $\frac{1}{Z_\beta}e^{-\beta\wick{E}}$.
\end{proposition}

\begin{proof}
	The fact that $\expt{\wick{E}(\eta)\, \B \phi(\eta)}=0$ is detailed in \cite[Theorem 3.1]{Ci99},
	and infinitesimal invariance of Gibbs density can be obtained by a completely analogous computation.
    By means of \eqref{infinitesimal}, one can straightforwardly check that the constant densities $\rho_t\equiv \ \wick{E}$
    or $\frac{1}{Z_\beta}e^{-\beta\wick{E}}$ solve the Fokker-Planck equation \eqref{fpe} for $\alpha=0$ in the sense of
    \autoref{def:weaksolfp}.
	In order to apply \autoref{thm:flow} and deduce existence of a stationary solution to Euler equation,
	we are only left to verify suitable integrability conditions.
	
	Since $\wick{E}$ belongs to the second Wiener chaos of $\mu$, it has finite moments of all orders, as well as exponential moments:
	we already mentioned that $e^{-\beta\wick{E}}$ is integrable as soon as $\beta>-1$. This threshold can be deduced from the explicit
	Gaussian expression \eqref{normenergy} and the standard result \cite[Theorem 6.1]{Ja97}.
\end{proof}

Thanks to the integrability properties of $\wick{E}$ and $\frac{1}{Z_\beta}e^{-\beta\wick{E}}$ we just recalled,
\autoref{thm:lp} and \autoref{thm:flow} provide existence of solutions to the stochastic Euler equation \eqref{stocheuler}
and the associated Fokker-Planck equation with initial data $\mu_{\beta}$ also for $\alpha>0$.

However, $\wick{E}$ \emph{is not invariant for the Ornstein-Uhlenbeck generator} $\L$,
as one can verify with an elementary computation in Fourier series in the same fashion of the above ones.
The resulting flow is thus not stationary.
When $\beta>-\frac12$, by the decay estimate in \autoref{thm:lp} for the case $p=2$
we know that the solutions we have built converge for large time to the space white noise.
Since uniqueness results are not available, we cannot rule out existence of ``anomalous'' solutions with a different behavior.
As already remarked in the Introduction, just like uniqueness of weak solutions, convergence to equilibrium in this setting
remains a fascinating open problem.

\bibliographystyle{plain}

\begin{thebibliography}{10}
	
	\bibitem{AlRFHK79}
	S.~Albeverio, M.~Ribeiro~de Faria, and R.~H{\"o}egh-Krohn.
	\newblock Stationary measures for the periodic {E}uler flow in two dimensions.
	\newblock {\em J. Statist. Phys.}, 20(6):585--595, 1979.
	
	\bibitem{AlCr90}
	Sergio Albeverio and Ana~Bela Cruzeiro.
	\newblock Global flows with invariant ({G}ibbs) measures for {E}uler and
	{N}avier-{S}tokes two-dimensional fluids.
	\newblock {\em Comm. Math. Phys.}, 129(3):431--444, 1990.
	
	\bibitem{AlFlSi08}
	Sergio Albeverio, Franco Flandoli, and Yakov~G. Sinai.
	\newblock {\em S{PDE} in hydrodynamic: recent progress and prospects}, volume
	1942 of {\em Lecture Notes in Mathematics}.
	\newblock Springer-Verlag, Berlin; Fondazione C.I.M.E., Florence, 2008.
	\newblock Lectures given at the C.I.M.E. Summer School held in Cetraro, August
	29--September 3, 2005, Edited by Giuseppe Da Prato and Michael R\"{o}ckner.
	
	\bibitem{BaFlMo10}
	D.~Barbato, F.~Flandoli, and F.~Morandin.
	\newblock Uniqueness for a stochastic inviscid dyadic model.
	\newblock {\em Proc. Amer. Math. Soc.}, 138(7):2607--2617, 2010.
	
	\bibitem{BaBeFe14}
	David Barbato, Hakima Bessaih, and Benedetta Ferrario.
	\newblock On a stochastic {L}eray-{$\alpha$} model of {E}uler equations.
	\newblock {\em Stochastic Process. Appl.}, 124(1):199--219, 2014.
	
	\bibitem{BaBiFlMo13}
	David Barbato, Luigi~Amedeo Bianchi, Franco Flandoli, and Francesco Morandin.
	\newblock A dyadic model on a tree.
	\newblock {\em J. Math. Phys.}, 54(2):021507, 20, 2013.
	
	\bibitem{BeFlGuMa14}
	Lisa {Beck}, Franco {Flandoli}, Massimiliano {Gubinelli}, and Mario {Maurelli}.
	\newblock {Stochastic ODEs and stochastic linear PDEs with critical drift:
		regularity, duality and uniqueness}.
	\newblock {\em under revision by Electr. J. of Probab.}, page arXiv:1401.1530,
	Jan 2014.
	
	\bibitem{BeFe12}
	Hakima Bessaih and Benedetta Ferrario.
	\newblock Invariant measures of {G}aussian type for 2{D} turbulence.
	\newblock {\em J. Stat. Phys.}, 149(2):259--283, 2012.
	
	\bibitem{BeFe14}
	Hakima Bessaih and Benedetta Ferrario.
	\newblock Inviscid limit of stochastic damped 2{D} {N}avier-{S}tokes equations.
	\newblock {\em Nonlinearity}, 27(1):1--15, 2014.
	
	\bibitem{BeFe16}
	Hakima Bessaih and Benedetta Ferrario.
	\newblock Statistical properties of stochastic 2{D} {N}avier-{S}tokes equations
	from linear models.
	\newblock {\em Discrete Contin. Dyn. Syst. Ser. B}, 21(9):2927--2947, 2016.
	
	\bibitem{BeFl99}
	Hakima Bessaih and Franco Flandoli.
	\newblock {$2$}-{D} {E}uler equation perturbed by noise.
	\newblock {\em NoDEA Nonlinear Differential Equations Appl.}, 6(1):35--54,
	1999.
	
	\bibitem{BeMi12}
	Hakima Bessaih and Annie Millet.
	\newblock Large deviations and the zero viscosity limit for 2{D} stochastic
	{N}avier-{S}tokes equations with free boundary.
	\newblock {\em SIAM J. Math. Anal.}, 44(3):1861--1893, 2012.
	
	\bibitem{Bi13}
	Luigi~Amedeo Bianchi.
	\newblock Uniqueness for an inviscid stochastic dyadic model on a tree.
	\newblock {\em Electron. Commun. Probab.}, 18:no. 8, 12, 2013.
	
	\bibitem{BiMo17}
	Luigi~Amedeo Bianchi and Francesco Morandin.
	\newblock Structure function and fractal dissipation for an intermittent
	inviscid dyadic model.
	\newblock {\em Comm. Math. Phys.}, 356(1):231--260, 2017.
	
	\bibitem{Bi99}
	Patrick Billingsley.
	\newblock {\em Convergence of probability measures}.
	\newblock Wiley Series in Probability and Statistics: Probability and
	Statistics. John Wiley \& Sons, Inc., New York, second edition, 1999.
	\newblock A Wiley-Interscience Publication.
	
	\bibitem{BoEc12}
	Guido Boffetta and Robert~E. Ecke.
	\newblock Two-dimensional turbulence.
	\newblock In {\em Annual review of fluid mechanics. {V}olume 44, 2012},
	volume~44 of {\em Annu. Rev. Fluid Mech.}, pages 427--451. Annual Reviews,
	Palo Alto, CA, 2012.
	
	\bibitem{BrFeHo18}
	Dominic Breit, Eduard Feireisl, and Martina Hofmanov\'{a}.
	\newblock {\em Stochastically forced compressible fluid flows}, volume~3 of
	{\em De Gruyter Series in Applied and Numerical Mathematics}.
	\newblock De Gruyter, Berlin, 2018.
	
	\bibitem{BrFlMa16}
	Zdzis{\l}aw Brze\'{z}niak, Franco Flandoli, and Mario Maurelli.
	\newblock Existence and uniqueness for stochastic 2{D} {E}uler flows with
	bounded vorticity.
	\newblock {\em Arch. Ration. Mech. Anal.}, 221(1):107--142, 2016.
	
	\bibitem{BrPe01}
	Zdzis{\l}aw Brze\'{z}niak and Szymon Peszat.
	\newblock Stochastic two dimensional {E}uler equations.
	\newblock {\em Ann. Probab.}, 29(4):1796--1832, 2001.
	
	\bibitem{Ci99}
	Fernanda Cipriano.
	\newblock The two-dimensional {E}uler equation: a statistical study.
	\newblock {\em Comm. Math. Phys.}, 201(1):139--154, 1999.
	
	\bibitem{CoGHVi14}
	Peter Constantin, Nathan Glatt-Holtz, and Vlad Vicol.
	\newblock Unique ergodicity for fractionally dissipated, stochastically forced
	2{D} {E}uler equations.
	\newblock {\em Comm. Math. Phys.}, 330(2):819--857, 2014.
	
	\bibitem{DPFl10}
	G.~Da~Prato and F.~Flandoli.
	\newblock Pathwise uniqueness for a class of {SDE} in {H}ilbert spaces and
	applications.
	\newblock {\em J. Funct. Anal.}, 259(1):243--267, 2010.
	
	\bibitem{DPFlPrRo13}
	G.~Da~Prato, F.~Flandoli, E.~Priola, and M.~R\"{o}ckner.
	\newblock Strong uniqueness for stochastic evolution equations in {H}ilbert
	spaces perturbed by a bounded measurable drift.
	\newblock {\em Ann. Probab.}, 41(5):3306--3344, 2013.
	
	\bibitem{DPFlRoVe16}
	G.~Da~Prato, F.~Flandoli, M.~R\"{o}ckner, and A.~Yu. Veretennikov.
	\newblock Strong uniqueness for {SDE}s in {H}ilbert spaces with nonregular
	drift.
	\newblock {\em Ann. Probab.}, 44(3):1985--2023, 2016.
	
	\bibitem{DPDe03}
	Giuseppe Da~Prato and Arnaud Debussche.
	\newblock Ergodicity for the 3{D} stochastic {N}avier-{S}tokes equations.
	\newblock {\em J. Math. Pures Appl. (9)}, 82(8):877--947, 2003.
	
	\bibitem{DPFlRo17}
	Giuseppe {Da Prato}, Franco {Flandoli}, and Michael {R{\"o}ckner}.
	\newblock {Continuity equation in LlogL for the 2D Euler equations under the
		enstrophy measure}.
	\newblock {\em arXiv e-prints}, page arXiv:1711.07759, Nov 2017.
	
	\bibitem{DPZa14}
	Giuseppe Da~Prato and Jerzy Zabczyk.
	\newblock {\em Stochastic equations in infinite dimensions}, volume 152 of {\em
		Encyclopedia of Mathematics and its Applications}.
	\newblock Cambridge University Press, Cambridge, second edition, 2014.
	
	\bibitem{DLSz09}
	Camillo De~Lellis and L\'{a}szl\'{o} Sz\'{e}kelyhidi, Jr.
	\newblock The {E}uler equations as a differential inclusion.
	\newblock {\em Ann. of Math. (2)}, 170(3):1417--1436, 2009.
	
	\bibitem{Fe06}
	Charles~L. Fefferman.
	\newblock Existence and smoothness of the {N}avier-{S}tokes equation.
	\newblock In {\em The millennium prize problems}, pages 57--67. Clay Math.
	Inst., Cambridge, MA, 2006.
	
	\bibitem{FlGuPr10}
	F.~Flandoli, M.~Gubinelli, and E.~Priola.
	\newblock Well-posedness of the transport equation by stochastic perturbation.
	\newblock {\em Invent. Math.}, 180(1):1--53, 2010.
	
	\bibitem{FlGuPr11}
	F.~Flandoli, M.~Gubinelli, and E.~Priola.
	\newblock Full well-posedness of point vortex dynamics corresponding to
	stochastic 2{D} {E}uler equations.
	\newblock {\em Stochastic Process. Appl.}, 121(7):1445--1463, 2011.
	
	\bibitem{Fl15}
	Franco Flandoli.
	\newblock {\em Random perturbation of {PDE}s and fluid dynamic models}, volume
	2015 of {\em Lecture Notes in Mathematics}.
	\newblock Springer, Heidelberg, 2011.
	\newblock Lectures from the 40th Probability Summer School held in Saint-Flour,
	2010, \'{E}cole d'\'{E}t\'{e} de Probabilit\'{e}s de Saint-Flour.
	[Saint-Flour Probability Summer School].
	
	\bibitem{Fl18}
	Franco Flandoli.
	\newblock Weak vorticity formulation of 2{D} {E}uler equations with white noise
	initial condition.
	\newblock {\em Comm. Partial Differential Equations}, 43(7):1102--1149, 2018.
	
	\bibitem{FlLu17}
	Franco {Flandoli} and Dejun {Luo}.
	\newblock {$\rho$-white noise solution to 2D stochastic Euler equations}.
	\newblock {\em Probab. Theory Relat. Fields}, https://doi.org/10.1007/s00440-019-00902-8, 2019.
	
	\bibitem{FlLu18b}
	Franco {Flandoli} and Dejun {Luo}.
	\newblock {Convergence of transport noise to Ornstein-Uhlenbeck for 2D Euler
		equations under the enstrophy measure}.
	\newblock {\em to appear on Ann. Probab.}, see arXiv:1806.09332, Jun 2018.
	
	\bibitem{FlLu18a}
	Franco {Flandoli} and Dejun {Luo}.
	\newblock {Kolmogorov equations associated to the stochastic two dimensional Euler equations}.
	\newblock {\em SIAM J. Math. Anal.}, 51(3):1761--1791, 2019.
	
	\bibitem{FlLu19b}
	Franco {Flandoli} and Dejun {Luo}.
	\newblock {Energy conditional measures and 2D turbulence}.
	\newblock {\em arXiv e-prints}, page arXiv:1902.10072, Feb 2019.
	
	\bibitem{FlLu19a}
	Franco {Flandoli} and Dejun {Luo}.
	\newblock {Point vortex approximation for 2D Navier--Stokes equations driven by
		space-time white noise}.
	\newblock {\em arXiv e-prints}, page arXiv:1902.09338, Feb 2019.
	
	\bibitem{FlMaNe14}
	Franco Flandoli, Mario Maurelli, and Mikhail Neklyudov.
	\newblock Noise prevents infinite stretching of the passive field in a
	stochastic vector advection equation.
	\newblock {\em J. Math. Fluid Mech.}, 16(4):805--822, 2014.
	
	\bibitem{FlRo08}
	Franco Flandoli and Marco Romito.
	\newblock Markov selections for the 3{D} stochastic {N}avier-{S}tokes
	equations.
	\newblock {\em Probab. Theory Related Fields}, 140(3-4):407--458, 2008.
	
	\bibitem{GaGe19}
	Paul Gassiat and Benjamin Gess.
	\newblock Regularization by noise for stochastic {H}amilton-{J}acobi equations.
	\newblock {\em Probab. Theory Related Fields}, 173(3-4):1063--1098, 2019.
	
	\bibitem{GeMa18}
	Benjamin Gess and Mario Maurelli.
	\newblock Well-posedness by noise for scalar conservation laws.
	\newblock {\em Comm. Partial Differential Equations}, 43(12):1702--1736, 2018.
	
	\bibitem{GHSvVi15}
	Nathan Glatt-Holtz, Vladim\'{\i}r \v{S}ver\'{a}k, and Vlad Vicol.
	\newblock On inviscid limits for the stochastic {N}avier-{S}tokes equations and
	related models.
	\newblock {\em Arch. Ration. Mech. Anal.}, 217(2):619--649, 2015.
	
	\bibitem{GHVi14}
	Nathan~E. Glatt-Holtz and Vlad~C. Vicol.
	\newblock Local and global existence of smooth solutions for the stochastic
	{E}uler equations with multiplicative noise.
	\newblock {\em Ann. Probab.}, 42(1):80--145, 2014.
	
	\bibitem{Gr19}
	Francesco {Grotto}.
	\newblock {Stationary Solutions of Damped Stochastic 2-dimensional Euler's
		Equation}.
	\newblock {\em arXiv e-prints}, page arXiv:1901.06744, Jan 2019.
	
	\bibitem{GrRo19}
	Francesco {Grotto} and Marco {Romito}.
	\newblock {A Central Limit Theorem for Gibbsian Invariant Measures of 2D Euler
		Equation}.
	\newblock {\em arXiv e-prints}, page arXiv:1904.01871, Apr 2019.
	
	\bibitem{Ja97}
	Svante Janson.
	\newblock {\em Gaussian {H}ilbert spaces}, volume 129 of {\em Cambridge Tracts
		in Mathematics}.
	\newblock Cambridge University Press, Cambridge, 1997.
	
	\bibitem{Ju63}
	V.~I. Judovi\v{c}.
	\newblock Non-stationary flows of an ideal incompressible fluid.
	\newblock {\em \v{Z}. Vy\v{c}isl. Mat. i Mat. Fiz.}, 3:1032--1066, 1963.
	
	\bibitem{Kr75}
	Robert~H. Kraichnan.
	\newblock Remarks on turbulence theory.
	\newblock {\em Advances in Math.}, 16:305--331, 1975.
	
	\bibitem{KrRo05}
	N.~V. Krylov and M.~R\"{o}ckner.
	\newblock Strong solutions of stochastic equations with singular time dependent
	drift.
	\newblock {\em Probab. Theory Related Fields}, 131(2):154--196, 2005.
	
	\bibitem{KuSh12}
	Sergei Kuksin and Armen Shirikyan.
	\newblock {\em Mathematics of two-dimensional turbulence}, volume 194 of {\em
		Cambridge Tracts in Mathematics}.
	\newblock Cambridge University Press, Cambridge, 2012.
	
	\bibitem{Ku04}
	Sergei~B. Kuksin.
	\newblock The {E}ulerian limit for 2{D} statistical hydrodynamics.
	\newblock {\em J. Statist. Phys.}, 115(1-2):469--492, 2004.
	
	\bibitem{Ku06}
	Sergei~B. Kuksin.
	\newblock {\em Randomly forced nonlinear {PDE}s and statistical hydrodynamics
		in 2 space dimensions}.
	\newblock Zurich Lectures in Advanced Mathematics. European Mathematical
	Society (EMS), Z\"{u}rich, 2006.
	
	\bibitem{Kullback}
	Solomon Kullback.
	\newblock A lower bound for discrimination in terms of variation.
	\newblock {\em IEEE Trans. Inform. Theory}, 16:126--127, 1967.
	
	\bibitem{Li98}
	Pierre-Louis Lions.
	\newblock {\em Mathematical topics in fluid mechanics. {V}ol. 2}, volume~10 of
	{\em Oxford Lecture Series in Mathematics and its Applications}.
	\newblock The Clarendon Press, Oxford University Press, New York, 1998.
	\newblock Compressible models, Oxford Science Publications.
	
	\bibitem{MaBe02}
	Andrew~J. Majda and Andrea~L. Bertozzi.
	\newblock {\em Vorticity and incompressible flow}, volume~27 of {\em Cambridge
		Texts in Applied Mathematics}.
	\newblock Cambridge University Press, Cambridge, 2002.
	
	\bibitem{Ma11}
	Mario Maurelli.
	\newblock Wiener chaos and uniqueness for stochastic transport equation.
	\newblock {\em C. R. Math. Acad. Sci. Paris}, 349(11-12):669--672, 2011.
	
	\bibitem{Sc95}
	Steven Schochet.
	\newblock The weak vorticity formulation of the {$2$}-{D} {E}uler equations and
	concentration-cancellation.
	\newblock {\em Comm. Partial Differential Equations}, 20(5-6):1077--1104, 1995.
	
	\bibitem{Si87}
	Jacques Simon.
	\newblock Compact sets in the space {$L^p(0,T;B)$}.
	\newblock {\em Ann. Mat. Pura Appl. (4)}, 146:65--96, 1987.
	
	\bibitem{Ve80}
	A.~Ju. Veretennikov.
	\newblock Strong solutions and explicit formulas for solutions of stochastic
	integral equations.
	\newblock {\em Mat. Sb. (N.S.)}, 111(153)(3):434--452, 480, 1980.
	
\end{thebibliography}

\end{document}